\documentclass[12pt]{amsart}
\usepackage{amssymb}
\usepackage{amsthm,amsfonts}

\usepackage{amsfonts}
\usepackage{amsmath}
\usepackage[abbrev]{amsrefs}
\usepackage{enumerate}
\usepackage{color}
\usepackage[colorlinks,linkcolor=blue]{hyperref}
\usepackage[nameinlink]{cleveref}

\newtheorem{Theorem}{Theorem}[section]

\newtheorem{Lemma}[Theorem]{Lemma}
\newtheorem{Corollary}[Theorem]{Corollary}
\newtheorem{Proposition}[Theorem]{Proposition}

\newcommand\supp{\mathop{\rm supp}}

\theoremstyle{definition}
\newtheorem{Remark}[Theorem]{Remark}

\begin{document}
	
	\title{On K\"othe duals of Orlicz-Lorentz spaces}
	
	\author{Anna Kami\'{n}ska}
	\address{Department of Mathematical Sciences,
		The University of Memphis, TN 38152-3240}
	\email{kaminska@memphis.edu}
	
	\author{Hyung-Joon Tag}
	\address{Department of Mathematics Education, Dongguk University, Seoul, Republic of Korea}
	\email{hjtag4@gmail.com}
	\date{\today}
	
	\maketitle

	\begin{abstract}
		In this article, we study a number of properties of the K\"othe duals $\mathcal{M}_{\varphi,w}$ of Orlicz-Lorentz spaces. An explicit description of the order-continuous subspace of $\mathcal{M}_{\varphi,w}$ is provided. Moreover, the separability of these spaces are characterized by the growth condition $\Delta_2$. Consequently, the K\"othe dual space $\mathcal{M}_{\varphi,w}$ has the Radon-Nikod\'ym property if and only if the N-function at infinity $\varphi$ satisfies the appropriate $\Delta_2$-condition. The comparison between $\mathcal{M}_{\varphi,w}$ spaces is characterized via standard orders between Orlicz functions. As applications of these results, we provide sufficient conditions for M-embedded order-continuous subspaces of Orlicz-Lorentz spaces equipped with the Luxemburg norm and prove the existence of a unique norm-preserving extension on Orlicz-Lorentz spaces equipped with the Orlicz norm.
	\end{abstract}

\section{Introduction}

In this article, we investigate basic properties of the K\"othe duals $\mathcal{M}_{\varphi,w}$ of Orlicz-Lorentz spaces and their applications to M-embeddedness and the uniqueness of norm-preserving extensions of bounded linear functionals on Orlicz-Lorentz spaces. The Orlicz-Lorentz spaces appeared as a natural generalization of both Orlicz and Lorentz spaces \cite{K}. Since then, some geometrical properties e.g.  separability, rotundity   or nonsquareness \cite{K, FHK} of these spaces have been studied. However, comparing to Orlicz spaces,  many properties of Orlicz-Lorentz spaces have not been investigated   because the explicit description of their K\"othe duals was not available until later. Moreover, these K\"othe duals exhibit different behaviors from those in the case of Orlicz spaces. It is well-known that the modular functionals for both Orlicz spaces and their K\"othe duals have the same structure, so we can study those spaces within the same framework. However, as it turns out, the analogous claim does not hold for Orlicz-Lorentz spaces.

To elaborate this more, in 2013, the first author and Raynaud introduced a new modular $P_{\varphi,w}$ \cite{KR2} and proved, with Le\'snik, that the class of Banach function spaces defined by $P_{\varphi_*,w}$ is actually the K\"othe dual spaces of Orlicz-Lorentz spaces $\Lambda_{\varphi,w}$ \cite{KLR}. The modular  $\rho_{\varphi, w}(f)$ of Orlicz-Lorentz space is based on the decreasing rearrangement  $f^*$ of $f$, while the modular $P_{\varphi,w}(f)$ is defined via $f^*/v$ where $v$ is a decreasing function submajorized by a weight function $w$ (i.e. $v\prec w$), or the level function $f^0$ of $f$ with respect to $w$. This major difference shows us that the modular $\rho_{\varphi, w}$ exhibits a different behavior from the modular $P_{\varphi,w}$. As a matter of fact, while the modular $\rho_{\varphi, w}$ is orthogonally subadditive, the modular $P_{\varphi,w}$ is orthogonally superadditive, and so constructing suitable functions to verify certain properties of the space $\mathcal{M}_{\varphi,w}$ requires different approaches from Orlicz-Lorentz spaces. The space $\mathcal{M}_{\varphi,w}$ has also been studied in a spirit of abstract Lorentz spaces. For interested readers, we refer to \cite{KR}.

From the fact that the space $\mathcal{M}_{\varphi,w}$ is not an Orlicz-Lorentz space, the properties of this space needs to be studied separately. Such attempts have been done partially in the past. For instance, a sufficient condition for the order-continuity of $\mathcal{M}_{\varphi,w}$ has been provided in \cite{FLM}. In the same paper, it is shown that the space $\mathcal{M}_{\varphi,w}$ is never rotund unless $w$ is a constant function. This consequently shows that Orlicz-Lorentz spaces are never smooth if its complementary function satisfies the appropriate $\Delta_2$-condition. Moreover, the explicit description of the space $\mathcal{M}_{\varphi,w}$ allows us to explore geometrical and ideal structures of Orlicz-Lorentz spaces $\Lambda_{\varphi,w}$. It is well-known that the order-continuous subspaces of Orlicz-Lorentz spaces equipped with the Luxemburg norm are M-ideals \cite{KLT2}. Moreover, the Orlicz-Lorentz spaces with the diameter two properties can be characterized by the $\Delta_2$-condition \cite{KLT3}. For the diameter two properties, it was essential to compute the fundamental function of the space $\mathcal{M}_{\varphi,w}$ in full generality. Since the space $\mathcal{M}_{\varphi,w}$ has appeared recently, many properties have not been investigated yet. We attempt to study some of those in the article.   

The paper consists of seven parts. In sections 2 and 3 we recall necessary facts on Orlicz-Lorentz spaces and their K\"othe duals. In section 4, from an explicit description of the order-continuous subspace of $\mathcal{M}_{\varphi,w}$ (Theorem \ref{Order}), we show that the separability of the space $\mathcal{M}_{\varphi,w}$ is characterized via the $\Delta_2$-condition (Theorem \ref{th:copy}). This consequently provides a sufficient condition for a class of $\mathcal{M}_{\varphi,w}$ to have the Radon-Nikod\'ym property (Theorem \ref{th:RNPM}). In section 5, we characterize the inclusion operators between the spaces $\mathcal{M}_{\varphi,w}$ defined by different Orlicz functions (Theorem \ref{th:comparison}). In section 6, we identify an interval for the Orlicz norm on $\mathcal{M}_{\varphi,w}$ that attains its value (Theorem \ref{th:OrliczM}). As applications in section 7, we study the M-embeddedness of the order-continuous subspaces of Orlicz-Lorentz spaces (Theorem \ref{MembedS}) and the uniqueness of norm-preserving extensions of bounded linear functionals on Orlicz-Lorentz spaces equipped with the Orlicz norm (Theorem \ref{existext}, \ref{ext}, and \ref{everyext}).
 
\section{Preliminaries}
Here we recall some facts from the theory of Banach function lattices. Let us denote $(\Omega, \Sigma, \mu)$ as a measure space $\Omega$ where $\Sigma$ is a $\sigma$-algebra defined on $\Omega$ and $\mu$ a $\sigma$-finite measure. The space of all $\mu$-measurable functions over a measure space $\Omega$ is denoted by $L_0(\Omega)$. The notation ``$\mu$-a.e." stands for almost everywhere convergence with respect to a measure $\mu$, but we simply denote ``a.e." when $\mu$ is the Lebesgue measure $m$. A Banach space $(X, \|\cdot\|_X) \subset L_0(\Omega)$ is said to be a {\it Banach function lattice} if for any $f\in L_0(\Omega)$ and $g\in X$, if  $0 \leq |f| \leq |g|$ $\mu$-a.e. then $f\in X$ and $\|f\|_X \leq \|g\|_X$.  We say a Banach function lattice has the {\it Fatou property} if for every sequence $(f_n)_{n=1}^{\infty} \subset X$ such that $0 \leq |f_n| \uparrow |f|$ $\mu$-a.e. and $\sup\|f_n\|_X < \infty$ we have $f \in X$ and $\|f_n\|_X \rightarrow \|f\|_X$.

For a Banach function lattice $X$, a function $f \in X$ is said to be {\it order-continuous} if $\|f_n\|_X \downarrow 0$ for every sequence $(f_n)_{n=1}^{\infty}$ of measurable functions such that $f_n \leq |f|$ for every $n \in  \mathbb{N}$ and $f_n\downarrow 0$ $\mu$-a.e. \cite[Proposition 1.3.5]{BS}. In this article, the term {\it simple function} stands for a function that has finitely many values and support of finite measure.
The set of all order-continuous elements in $X$ is denoted by $X_a$. The closure of the set of all simple functions is denoted by $X_b$. In general, we have $X_a \subset X_b$ \cite[Theorem 1.3.11]{BS}. A measure $\mu$ is separable if there exists a countable family $\mathcal{F}$ of measurable subsets such that for every $\epsilon > 0$ and $E \in \Sigma$ with $\mu E < \infty$, there exists a measurable subset $A \in \mathcal{F}$ such that $\mu(A \Delta E) < \epsilon$, where $A \Delta E$ is the symmetric difference between $A$ and $E$. It is well-known that $X_a$ is separable if and only if the measure $\mu$ is separable \cite[Theorem 1.5.5]{BS}. Hence, under our assumption on the measure space in this article, the order-continuous Banach function lattices are always separable.

The {\it K\"othe dual} $X' \subset L_0(\Omega)$ of a Banach function lattice $X$ is the set of $\mu$-measurable functions such that for each $g \in X'$ the associated norm $\|g\|_{X'}$ defined by
\[
\|g\|_{X'} = \sup\left\{\left|\int_\Omega gf\right|: \|f\|_X \leq 1\right\}
\]
is finite. When the order-continuous subspace $X_a$ contains all simple functions, we have $X_a = X_b$ and the K\"othe dual $X' \simeq (X_a)^*$ \cite[Corollary 1.4.2]{BS}. For sequence spaces, $X_a = X_b$ always holds \cite[Theorem 1.3.13]{BS}. 

For $\lambda > 0$, the {\it distribution function} $d_f(\lambda)$ of $f \in L_0(\Omega)$ is defined by $d_f(\lambda) = \mu\{t \in \Omega: |f(t)| > \lambda\}$. If $d_f(\lambda) = d_g(\lambda)$ for all $\lambda > 0$, then we say $f$ and $g$ are {\it equimeasurable}. For $f \in L_0(\Omega)$, the function $f^*(t) = \inf\{\lambda > 0: d_f(\lambda) \leq t\}$  is called the {\it decreasing rearrangement} of $f$, and $f^*$ is a Lebesgue measurable function  on interval $[0,\mu(\Omega))$. For any sequence $x$, the decreasing rearrangement $x^*$ of $x$ is defined by $x^*(i) = \inf\{\lambda > 0: d_x(\lambda) < i\}$, where $\mu$ is the counting measure on $\mathbb{N}$. In this article, the term ``decreasing" refers to non-increasing. 
  The functions $f$ and $f^*$ are equimeasurable, and the same fact also holds for  sequences $x$ and $x^*$. 
  
 Let $f$ and $g$ be   Lebesgue measurable functions  on $I=[0,\gamma)$, $\gamma \le \infty$. We say that $f$ is  {\it submajorized} by $g$ if $\int_0^t f \leq \int_0^t g$ for every $t \in I$. By using the Hardy-Littlewood-Pol\'ya relation \cite[Definition 2.3.5]{BS}, we denote this submajorization by $f \prec g$. We will also use this notation for sequences analogously. The following facts will be useful later.

\begin{Lemma} {\bf (Hardy's Lemma)} \cite[Proposition 2.3.6] {BS}
	\label{lem:hardy}
	Let $f_1$ and $f_2$ be non-negative Lebesgue measurable functions on $[0, \gamma)$, $0<\gamma\le \infty$,  and suppose $\int_{0}^{t} f_1(s)ds \leq \int_{0}^{t} f_2(s)ds$ for all $t \in [0, \gamma)$. Let $g$ be any non-negative decreasing function on $[0,\gamma)$. Then $\int_{0}^{\gamma} f_1(s)g(s)ds \leq \int_{0}^{\gamma} f_2(s)g(s)ds$. 
\end{Lemma}

\begin{Lemma}\cite[pg 67]{KPS}\label{convto0dec}
	Let $f \in L_0(\Omega)$ be such that $d_{f}(\lambda) < \infty$ for all $\lambda >0$ and $(f_n)$ be a sequence of measurable functions such that $|f_n| \leq |f|$ $\mu$-a.e. for every $n \in \mathbb{N}$. If $f_n \rightarrow f$ $\mu$-a.e. then $(f - f_n)^* \rightarrow 0$ a.e. The same result remains true for sequences.  
\end{Lemma}

For a Banach function lattice $X$, if all equimeasurable functions $f,g \in X$ satisfies $\|f\|_X = \|g\|_X$, then we say $X$ is {\it rearrangement-invariant} (r.i.). Not only Orlicz-Lorentz spaces but also their K\"othe dual spaces are known to be r.i. and to have the Fatou property \cite[Theorem 4.7]{KR}. The fundamental function of r.i. Banach function lattice $X$ is defined by $\phi_X(t)= \|\chi_{E_t}\|_X$ where $E_t$ is a set of measure $t \in [0, \infty)$. For a r.i. Banach function space $X$ over a nonatomic $\sigma$-finite measure space, if $\lim_{t \rightarrow 0+}\phi_X(t) = 0$, then we have $X_a = X_b$ \cite[Theorem 2.5.5]{BS}. For a r.i. Banach sequence space, this relationship always holds \cite[Theorem 2.5.4]{BS}. We refer to \cite{BS, Z} for more details on the theory of Banach function lattices.

Throughout the article, the notation $\supp f$ stands for the support of a $\mu$-measurable function $f$. When a Banach space $X$ is isometrically isomorphic to another Banach space $Y$, we will use the notation $X \simeq Y$. From now on, unless specified, we will only consider $\Omega = I = [0, \gamma), \gamma \leq \infty$ with the Lebesgue measure $\mu = m$ and $\Omega = \mathbb{N}$ with the counting measure. Accordingly, we denote $L_0 = L_0(I)$ for Lebesgue measurable functions on $I$, and $\ell_0 = L_0(\mathbb{N})$ for sequences.

\section{Orlicz-Lorentz spaces and their K\"othe duals}

In this section, we provide basic information on Orlicz-Lorentz spaces and their K\"othe dual spaces $\mathcal{M}_{\varphi,w}$. Also, a few facts about level functions will be presented here.

An {\it Orlicz function} $\varphi: \mathbb{R}^+ \rightarrow [0, \infty)$ is a convex function such that $\varphi(0) = 0$ and $\varphi(u) > 0$ for all $u > 0$.  The {\it complementary function} $\varphi_*$ of an Orlicz function $\varphi$ is defined by $\varphi_*(v) = \sup\{uv - \varphi(u) : u > 0\}$,  $ v\ge 0$. Even though the function $\varphi_*$ is also convex and $\varphi_*(0) = 0$, we can have $\varphi_*(v) = 0$ for some $v > 0$ and $\varphi_*(v) = \infty$ for some $v < \infty$. For instance, the complementary function $\varphi_*$ of $\varphi(u) = u$ has value $0$ on the interval $[0,1]$ and $\infty$ outside of $[0,1]$. An Orlicz function $\varphi$ and its complementary function $\varphi_*$ satisfies {\it Young's inequality}, that is, $uv \leq \varphi(u) + \varphi_*(v)$ for all $u,v \in \mathbb{R}^+$. An Orlicz function $\varphi$ is said to be an {\it N-function at zero} if $\lim_{u \rightarrow 0} \frac{\varphi(u)}{u}= 0$ and an {\it N-function at infinity} if $\lim_{u \rightarrow \infty} \frac{\varphi(u)}{u}= \infty$. An Orlicz function is $N$-function at infinity if and only if its complementary function is finite \cite{KLT3}. When an Orlicz function $\varphi$ is both $N$-function at zero and at infinity, we say $\varphi$ is an {\it N-function}. 
An Orlicz function $\varphi$ satisfies  {\it $\Delta_2^0$-condition} if there exist $K > 2$ and $u_0 > 0$ such that $\varphi(2u) \leq K \varphi(u)$ for all $0\le u \leq u_0$. Analogously, $\varphi$ satisfies  {\it $\Delta_2^{\infty}$-condition} if there exist $K > 2$ and $u_0 \geq 0$ such that  $\varphi(2u) \leq K \varphi(u)$ for all $u \geq u_0$. If $\varphi$ satisfies  both $\Delta_2^0$- and $\Delta_2^{\infty}$-conditions, we say that $\varphi$ satisfies the {\it $\Delta_2$-condition}. In this article, the ``appropriate" $\Delta_2$-condition means  $\Delta_2^\infty$-condition for $\gamma < \infty$,  $\Delta_2$-condition for  $\gamma = \infty$, and   $\Delta_2^0$-condition for the sequence spaces. We have the following lemma that will be useful in Section 4. 

\begin{Lemma} \cite[Theorem 1.13]{Chen}  \label{le:C} 
	An Orlicz function $\varphi$ satisfies the $\Delta_2$- 
	(resp., $\Delta_2^\infty$-; $\Delta_2^0$-) condition if and only if there exist $l>1$ and $K>1$ (resp.,  $l>1, K>1$,  $u_0\ge 0$; $l>1, K>1, u_0 > 0$) such that $\varphi(lu) \leq K \varphi(u)$ for all $u\ge 0$ (resp., $u\ge u_0$; $0\le u\le u_0$). 
\end{Lemma}

Let the {\it weight function} $w: I \rightarrow [0, \gamma)$ be a decreasing and locally integrable function. Then $W(t) := \int_0^t w < \infty$ for all $t\in I$.  In this article we always assume that $W(\infty) = \int_0^\infty w =\infty$ . The {\it modular} for Orlicz-Lorentz function spaces $\rho_{\varphi,w}(\cdot): L_0 \rightarrow [0, \infty]$ on $f \in L_0$ is defined by
\[
\rho_{\varphi,w}(f) = \int_I \varphi(f^*)w,
\]  
and
the  modular for Orlicz-Lorentz sequence spaces $\alpha_{\varphi,w}(\cdot): \ell_0 \rightarrow [0, \infty]$ on $x \in \ell_0$ by
\[
\alpha_{\varphi,w}(x) = \sum_{i=1}^\infty \varphi(x^*(i))w(i).
\] 
The modular $\rho_{\varphi,w}$ is convex and orthogonally subadditive, that is for $f \wedge g = 0$, $\rho_{\varphi, w}(f + g) \leq \rho_{\varphi,w}(f) + \rho_{\varphi,w}(g)$. If $0 \leq|f| \leq |g|$ a.e., then $\rho_{\varphi,w}(f) \leq \rho_{\varphi,w}(g)$. Similarly, these facts are also true for the modular $\alpha_{\varphi,w}$.

The {\it Orlicz-Lorentz function} and {\it sequence space} $\Lambda_{\varphi,w}$ and $\lambda_{\varphi,w}$ are defined respectively by
\begin{align*}
	\Lambda_{\varphi,w} &= \{f \in L_0 : \rho_{\varphi, w}(kf) < \infty\,\,\, \text{for some} \,\,\, k > 0\},\\
	\lambda_{\varphi,w} &= \{x \in \ell_0 : \alpha_{\varphi, w}(kx) < \infty\,\,\, \text{for some} \,\,\, k > 0\}.
\end{align*}
Notice that we get the Orlicz space $L_{\varphi}$ (resp. $\ell_\varphi$) if $w \equiv 1$ and the Lorentz space $\Lambda_{p, w}$ for $1 \leq p < \infty$ (resp. $\lambda_{p,w}$) if  $\varphi(u) = u^p$.

We consider the {\it Luxemburg norm} $\|\cdot\|_{\varphi,w}$ on $\Lambda_{\varphi, w}$ defined by
\[
	\|f\|_{\varphi,w} = \inf\left\{\epsilon > 0 : \rho_{\varphi,w}\left(\frac{f}{\epsilon}\right) \leq 1\right\},
\]
and the {\it Orlicz norm} $\|\cdot\|_{\varphi,w}^0$ on $\Lambda_{\varphi, w}$ defined by
\[	
	\|f\|_{\varphi,w}^0 = \inf_{k > 0}\left(\frac{1}{k}(1 + \rho_{\varphi, w}(kf))\right). 
\]
The analogous norms for the sequence space $\lambda_{\varphi,w}$ can be defined by replacing $\rho_{\varphi, w}$ with $\alpha_{\varphi,w}$. These norms are known to be equivalent to each other, precisely, 
\begin{align*}
\|f\|_{\varphi,w} &\leq \|f\|_{\varphi,w}^0 \leq 2\|f\|_{\varphi,w}, \\
\|x\|_{\varphi,w} &\leq \|x\|_{\varphi,w}^0 \leq 2\|x\|_{\varphi,w}. 
\end{align*}
From now on, we denote an Orlicz-Lorentz function (resp. sequence) space equipped with the Luxemburg norm by $\Lambda_{\varphi,w}$ (resp. $\lambda_{\varphi,w}$) and with the Orlicz norm by $\Lambda_{\varphi,w}^0$ (resp. $\lambda_{\varphi,w}^0$), respectively. For more information on Orlicz-Lorentz spaces and their properties, we refer to \cite{K}.

An explicit description of the K\"othe dual spaces $\mathcal{M}_{\varphi,w}$ and $\mathfrak{m}_{\varphi,w}$ of Orlicz-Lorentz function and sequence spaces has been given  and studied in \cite{FLM,KLR}. In \cite{KR2}, the more abstract and general approach clarifies and completes the studies. There are two modulars $P_{\varphi,w}$ and $Q_{\varphi,w}$ for $f \in L_0$ defined by 
\[
P_{\varphi,w}(f) = \inf\left\{\int_I \varphi\left(\frac{f^*}{v}\right)v: v \prec w, v \downarrow \right\}, \, \text{and} \,\,\, Q_{\varphi,w}(f) = \int_I \varphi\left(\frac{(f^*)^0}{w}\right)w = \int_I \varphi\left(\frac{f^*}{w^{f^*}}\right)w^{f^*},
\]	
and $\mathfrak{p}_{\varphi,w}$ and $\mathfrak{q}_{\varphi,w}$ for $x \in \ell_0$ defined by
\begin{align*}
\mathfrak{p}_{\varphi,w}(x) &= \inf\left\{\sum_{i=1}^{\infty} \varphi\left(\frac{x^*(i)}{v(i)}\right)v(i): v \prec w, v \downarrow \right\}, \,\text{and}\\ 
\mathfrak{q}_{\varphi,w}(x) &= \sum_{i=1}^{\infty} \varphi\left(\frac{(x^*)^0(i)}{w(i)}\right)w(i) = \sum_{i=1}^{\infty} \varphi\left(\frac{x^*(i)}{w^{x^*}(i)}\right)w^{x^*}(i). 
\end{align*}

Here $f^0$ (resp.  $x^0$) denotes the level function of $f$ (resp. $x$) in the sense of Halperin \cite{Hal} and $w^{f^*}$ (resp. $w^{x^*}$) is the inverse level function with respect to $f^*$ (resp. $x^*$) \cite{KLR, KR2}. If an Orlicz function is an $N$-function, then $P_{\varphi,w}(f) = Q_{\varphi,w}(f)$ (resp. $\mathfrak{p}_{\varphi,w}(x) = \mathfrak{q}_{\varphi,w}(x)$) for every $f \in L_0$ (resp. $x \in L_0$). However, without the assumption, we only know $P_{\varphi,w}(f) \leq Q_{\varphi,w}(f)$ (resp. $\mathfrak{p}_{\varphi,w}(f) \leq \mathfrak{q}_{\varphi,w}(f)$) in general \cite[Theorem 8.9]{KR2}. With these modulars, the function and sequence spaces $\mathcal{M}_{\varphi,w}$ and $\mathfrak{m}_{\varphi,w}$ \cite[Corollary 7.7]{KR2} are given by
\begin{align*}
	\mathcal{M}_{\varphi,w} &= \{f \in L_0 : P_{\varphi, w}(kf) < \infty\,\,\, \text{for some} \,\,\, k > 0\} = \{f \in L_0 : Q_{\varphi, w}(kf) < \infty\,\,\, \text{for some} \,\,\, k > 0\},\\
	\mathfrak{m}_{\varphi,w} &= \{x \in \ell_0 : \mathfrak{p}_{\varphi, w}(kx) < \infty\,\,\, \text{for some} \,\,\, k > 0\} = \{x \in \ell_0 : \mathfrak{q}_{\varphi, w}(kx) < \infty\,\,\, \text{for some} \,\,\, k > 0\}.
\end{align*}

The Luxemburg norm $\|\cdot\|_{\mathcal{M}_{\varphi,w}}$ and the Orlicz norm $\|\cdot\|_{\mathcal{M}_{\varphi,w}}^0$ \cite[Theorem 8.8]{KR2} are defined by
\begin{align*}
	\|f\|_{\mathcal{M}_{\varphi,w}} &= \inf\left\{\epsilon > 0 : P_{\varphi,w}\left(\frac{f}{\epsilon}\right) \leq 1\right\} = \inf\left\{\epsilon > 0 : Q_{\varphi,w}\left(\frac{f}{\epsilon}\right) \leq 1\right\},\\
	\|f\|_{\mathcal{M}_{\varphi,w}}^0 &= \inf_{k > 0}\left(\frac{1}{k}(1 + P_{\varphi,w}(kf))\right) = \inf_{k > 0}\left(\frac{1}{k}(1 + Q_{\varphi,w}(kf))\right).
\end{align*}
The analgous norms for the sequence space $\mathfrak{m}_{\varphi,w}$ can be defined by replacing $P_{\varphi,w}$ (resp. $Q_{\varphi,w}$) with $\mathfrak{p}_{\varphi,w}$ (resp. $\mathfrak{q}_{\varphi,w}$). From now on, we denote the function (resp. sequence) space $\mathcal{M}_{\varphi,w}$ (resp. $\mathfrak{m}_{\varphi,w}$) equipped with the Luxemburg norm and the space $\mathcal{M}_{\varphi,w}^0$ (resp. $\mathfrak{m}_{\varphi,w}^0$) equipped with the Orlicz norm, respectively. We recall the relationship between the Orlicz-Lorentz spaces and the spaces $\mathcal{M}_{\varphi,w}$.

\begin{Theorem}\label{th:KLR}\cite{KLR, KR2, FLM}
	Let $\varphi$ be an Orlicz function and $w$ be a locally integrable, non-negative decreasing weight function on $I = [0, \gamma)$. Then the K\"othe dual spaces of the Orlicz-Lorentz spaces $\Lambda_{\varphi, w}$ and $\Lambda_{\varphi, w}^0$ are given by $\Lambda_{\varphi, w}' \simeq \mathcal{M}_{\varphi_*,w}^0$ and $(\Lambda_{\varphi, w}^0)' \simeq \mathcal{M}_{\varphi_*,w}$. The similar facts are also true for the sequence spaces.
	
\end{Theorem}

As we have seen earlier, working with the modulars $Q_{\varphi,w}$ and $\mathfrak{q}_{\varphi,w}$ involves Halperin's level functions. Let $f \in L_0$ where $f \geq 0$ and set $F(t_1, t_2) = \int_{t_1}^{t_2} f(x)dx$, and $W(t_1, t_2) = \int_{t_1}^{t_2} w(x)dx$, where $0\le t_1\le t_2\le \gamma$. An interval $(a, b]$ is said to be a {\it level interval} of $f$ if $\frac{F(a,t)}{W(a,t)} \leq \frac{F(a,b)}{W(a,b)}$ for every $t \in (a, b]$. We say a level interval is {\it maximal} if it is not contained in other level intervals. Then the {\it level function} $f^0$ of $f$ is defined by
\[ f^0(t) = \begin{cases} 
	\frac{F(a, b)}{W(a,b)}w(t), & t \in (a,b)\,\,\, \text{for each maximal level interval}\,\,\, (a,b),  \\
	f(t), & t \notin \cup(a,b). 
\end{cases}
\]
The level sequences can be also defined in a similar way. There are countably many maximal level intervals, and if $f^0$ is a non-negative, locally integrable function, we can express $f^0 = Dw$ where $D$ is a decreasing function \cite{L, Hal, FLM}. 

Let $w$ be a weight function. A non-negative function $L$ on $I$ is {\it $w$-affine} if $L(t) = cW(t) + d$ for $c, d \in \mathbb{R}$. An increasing positive function $F$ on $I$ is said to be {\it $w$-concave} if $L(t) \leq F(t)$ for all $t \in (a, b)$, where $L$ is a $w$-affine function such that $L(a) = F(a)$ and $L(b) = F(b)$ for all $0 \leq a < b < \gamma$. If a function $F(t) = \int_0^t f$ is $w$-concave, we can express the function $f$ as $f = Dw$ with some decreasing function $D$ \cite[Theorem 2.1.(iii)]{FLM}. For a positive function $f$, consider $FW(t) = \int_{0}^{t}f(s)w(s)ds$ for $t \in I$ and its least $w$-concave majorant denoted by $F^b$. Then there exists $f^s$ such that $F^b(t) = \int_0^t f^s(s)w(s)ds$ because of the $w$-concavity of $F^b$. The function $f^s$ is called the level function in the sense of Sinnamon \cite{Sinn1}. It is well-known that if $0 \leq f \leq g$ a.e. then $f^s \leq g^s$ a.e., and if $0 \leq f_n \uparrow f$ a.e. then $f_n^s \uparrow f^s$ a.e \cite{Sinn2}. We mention that such monotonicity and convergence properties were not explicitly stated for Halperin's level functions in the literature. However, the following relationship stated in \cite{FLM} between Halperin's level functions and Sinnamon's level functions

\begin{equation}\label{eq:important}
\frac{f^0}{w} = \left(\frac{f}{w}\right)^s a.e.
\end{equation}
provides us with the similar properties for the level functions defined by Halperin. Consequently,

\begin{Proposition}\label{prop:Halmono}
	Let $f, f_n, g\in L_0$ be non-negative, locally integrable functions on $I$ and $w$ be a weight function. If $0\leq f \leq g$ a.e., then $f^0 \leq g^0$ a.e. If $0 \leq f_n \uparrow f$ a.e., then $f_n^0 \uparrow f^0$ a.e. The analogous statement also holds for sequences.
\end{Proposition}

Since the level function takes average values of $f$ over each maximal level intervals with respect to a weight function $w$, the corresponding values of $f$ and $f^0$ may differ on these intervals. Hence, it is natural to ask when these two functions can be the same. Based on our observation so far, we can identify a certain type of functions which level functions are identical to themselves. These functions turn out to be useful in the next section.

\begin{Proposition}\label{prop:declevel}
	Let $f \in L_0$ be a decreasing, non-negative function and $w$ be a weight function. Then, $(fw)^0 = fw$ a.e.  The analogous statement also holds for sequences.  
\end{Proposition} 

\begin{proof}
	
	First, we consider a decreasing simple function $f = \sum_{i=1}^{n} c_n \chi_{[t_{i-1}, t_{i})}$ where $t_0 = 0$ and $c_1 > c_2 > \cdots > c_n > 0$. Denote $FW(t_{i-1}, t_i) = \int_{t_{i-1}}^{t_i}fw$ and $W(t_{i-1}, t_i) = \int_{t_{i-1}}^{t_i}w$, $i=1,\dots,n$. We claim that each subinterval $(t_{i-1}, t_i)$ is a maximal level interval. Observe that
	\[
	\frac{FW(0, t)}{W(0, t)} = \frac{c_1W(0, t)}{W(0, t)} = c_1
	\]
	for every $t \in (0, t_1)$. So the interval $(0, t_1)$ is a level interval. Moreover, for every $t \in (t_{i-1}, t_i)$, where $i = 2, 3, \dots n$, notice that
	\[
	\frac{FW(0, t)}{W(0, t)} = \frac{c_1W(0, t_1)  +  \cdots + c_iW(t_{i-1}, t)}{W(0, t)} < c_1.
	\]
	 Hence, there is no level interval that contains $(0, t_1)$, and so the level interval $(0, t_1)$ is maximal. By the similar procedure, we can also show that each subinterval $(t_{i-1}, t_i)$ is a maximal level interval. 
	  Thus, from the fact that
	\[
	(fw)^0(t) = \frac{FW(t_{i-1}, t_i)}{W(t_{i-1}, t_i)}w(t) = c_i w(t) \,\,\, \text{for every} \,\,\, t \in (t_{i-1}, t_i),
	\]
	we obtain
	\[
	\frac{(fw)^0}{w} = \sum_{i=1}^{n} c_i \chi_{[t_{i-1}, t_{i})} = f \ \  a.e.
	\] 
	Now, let $f \in L_0$ be a non-negative, decreasing function. Then there exists a sequence of decreasing simple functions $(f_m)_{m=1}^{\infty}$ such that $f_m \uparrow f$ a.e. Since $f_m w \uparrow f w$ a.e., we see that $(f_m w)^0 = f_m w\uparrow fw$ a.e. But then, in view of Proposition \ref{prop:Halmono}, we have $(f_m w)^0 \uparrow (fw)^0$ a.e. Therefore, $(fw)^0 = fw$ a.e. 
\end{proof}

In this article, we always assume that $\varphi$ is an Orlicz function and that $w$ is a weight function. If there are results that only hold for $N$-functions (at infinity), we will state them explicitly.

\section {separability of $\mathcal{M}_{\varphi,w}$ and $\mathfrak{m}_{\varphi,w}$}
In this section, we study the separabilty of the spaces $\mathcal{M}_{\varphi,w}$ and $\mathfrak{m}_{\varphi,w}$. The results in this section will play an important role to study of M-embeddedness of Orlicz-Lorentz spaces. We mention that the same facts also hold for $\mathcal{M}_{\varphi,w}^0$ and $\mathfrak{m}_{\varphi,w}^0$ due to the equivalence of the Luxemburg  and the Orlicz norms \cite{KLR}. 

For convenience, we provide the explicit computation of $\frac{(c\chi_{[0, s)})^0}{w}$, where $c > 0$ and $s \in I$. Since the mapping $t \mapsto \frac{t}{W(t)}$ is increasing, notice that

\begin{equation}\label{eq:charlevel}
	\frac{(c\chi_{(0, s)})^0(t)}{w(t)} = \frac{c\cdot s}{W(s)}\chi_{(0, s)}(t).  	
\end{equation}

Using (\ref{eq:charlevel}) and the definition of $\|\cdot\|_{\mathcal{M}_{\varphi,w}}$, we can directly compute the fundamental function $\phi_{\mathcal{M}_{\varphi,w}}$.
	\begin{Lemma} \cite[Proposition 4.9]{KR} {\rm{(}c.f. \cite[Theorem 2.3]{KLT3} and \cite[Proposition 2.2]{KT}\rm{)}}\label{lem:Mfund}
	The fundamental function $\phi_{\mathcal{M}_{\varphi,w}}$ is expressed by
		\[
		\phi_{\mathcal{M}_{\varphi,w}}(t) = \frac{t}{W(t)}\bigg/\frac{1}{\varphi^{-1}(1/W(t))},
		\]
		where $\varphi^{-1}$ is the inverse function of $\varphi$. Consequently, $\lim_{t \rightarrow 0+}\phi_{\mathcal{M}_{\varphi,w}}(t) = 0$.
	\end{Lemma}

Now, we observe the relationship between $f \in \mathcal{M}_{\varphi,w}$ and its distribution function $d_f(\lambda)$.
\begin{Proposition}\label{prop:distfin}
	If $f \in \mathcal{M}_{\varphi,w}$, then the distribution function $d_f(\lambda) < \infty$ for all $\lambda >0$. The similar result holds for the sequence space $\mathfrak{m}_{\varphi,w}$.
\end{Proposition}

\begin{proof}
	Let $f \in \mathcal{M}_{\varphi,w}$. Then $Q_{\varphi,w}(k f) < \infty$ for some $k > 0$. Since $f^*$ is equimeasurable to $f$, we have $d_{f^*}(\lambda) = d_{f}(\lambda)$ for all $\lambda > 0$. Given $\lambda >0$, define $g = \lambda \chi_{[0, a)}$ where $a = d_{f^*}(\lambda)$. Notice that $g \leq f^*$ a.e. Then $\frac{g^0}{w} \leq \frac{(f^*)^0}{w}$ a.e. by Proposition \ref{prop:Halmono}. From (\ref{eq:charlevel}), we see that
	
	\begin{multline}\label{eq:distinf}
		\varphi\left(\frac{k\lambda a}{W(a)}\right) W(a) = Q_{\varphi,w}(kg) = \int_I \varphi\left(\frac{k \lambda(\chi_{(0, a)})^0}{w}\right)w\\
		\leq \int_I \varphi\left(\frac{k(f^*)^0}{w}\right)w = Q_{\varphi,w}(k f) < \infty.
	\end{multline}
	
	\normalsize
	Since the function $t \mapsto \frac{t}{W(t)}$ is increasing, there exists $a_0 \in (0, a)$ such that $ M_0  = \frac{a_0}{W(a_0)}\leq \frac{a}{W(a)}$. Then we have 
	
	\[
	\varphi(k\lambda M_0)W(a)\leq \varphi\left(\frac{k\lambda a}{W(a)}\right) W(a).  
	\]  
	
	Now suppose that there exists $\lambda > 0$ such that $d_{f}(\lambda) = d_{f^*}(\lambda) = \infty$. From the assumption that $W(\infty) = \infty$, we see that $\varphi(k\lambda M_0)W(a) = \infty$. This is a contradiction because $Q_{\varphi,w}(k f) < \infty$. Therefore, $d_{f}(\lambda) < \infty$ for all $\lambda > 0$. 
\end{proof}

Also we need to examine the relationship between the norm $\|\cdot\|_{\mathcal{M}_{\varphi,w}}$ and the modular $P_{\varphi,w}$.

\begin{Proposition}\label{lem:basiclux}
	If $\|f\|_{\mathcal{M}_{\varphi,w}} \leq 1$, then $P_{\varphi,w}(f) \leq Q_{\varphi,w}(f) \leq  \|f\|_{\mathcal{M}_{\varphi,w}}$. The analogous statement for the space $\mathfrak{m}_{\varphi,w}$ also holds. 
\end{Proposition}

\begin{proof}
	Let $f \in \mathcal{M}_{\varphi,w}$. From the definition of the Luxemburg norm, there exists a sequence of real numbers $(\lambda_n) \downarrow \|f\|_{\mathcal{M}_{\varphi,w}}$ such that $Q_{\varphi,w}\left(\frac{f}{\lambda_n}\right) \leq 1$. Then $\frac{f^*}{\lambda_n} \uparrow \frac{f^*}{\|f\|_{\mathcal{M}_{\varphi,w}}}$ a.e. Furthermore, $\frac{(f^*)^0}{\lambda_n} \uparrow \frac{(f^*)^0}{\|f\|_{\mathcal{M}_{\varphi,w}}}$ a.e. from Proposition \ref{prop:Halmono}. So $\varphi\left(\frac{(f^*)^0}{\lambda_n}\right)w \uparrow \varphi\left(\frac{(f^*)^0}{\|f\|_{\mathcal{M}_{\varphi,w}}}\right)w$ a.e. By the Monotone Convergence Theorem,  
	\[
	Q_{\varphi,w}\left(\frac{f}{\|f\|_{\mathcal{M}_{\varphi,w}}}\right) = \int_I \varphi\left(\frac{(f^*)^0}{\|f\|_{\mathcal{M}_{\varphi,w}}}\right)w = \lim_{n\rightarrow \infty} \int_I \varphi\left(\frac{(f^*)^0}{\lambda_n}\right)w \leq 1.
	\]
Hence by the convexity of $\varphi$, we obtain 
	\[
	\frac{P_{\varphi,w}(f)}{\|f\|_{\mathcal{M}_{\varphi,w}}} \leq \frac{Q_{\varphi,w}(f)}{\|f\|_{\mathcal{M}_{\varphi,w}}}\leq Q_{\varphi,w}\left(\frac{f}{\|f\|_{\mathcal{M}_{\varphi,w}}}\right) \leq 1.
	\]  
\end{proof}

\begin{Proposition}\label{modnormeq}
	Let $(f_n)_{n=1}^{\infty}$ be a sequence in $\mathcal{M}_{\varphi,w}$. Then, $\lim_{n \rightarrow \infty} \|f_n\|_{\mathcal{M}_{\varphi,w}} = 0$ if and only if $\lim_{n \rightarrow \infty} P_{\varphi,w}(k f_n) = \lim_{n \rightarrow \infty} Q_{\varphi,w}(k f_n) = 0$ for every $k > 0$. The analogous statement holds for the space $\mathfrak{m}_{\varphi,w}$.
\end{Proposition}

\begin{proof}
	Suppose that $\lim_{n \rightarrow \infty} \|f_n\|_{\mathcal{M}_{\varphi,w}} = 0$ and fix $k > 0$. Then for every $c >\max\{1, \frac{1}{k}\}$, there exists $N(c) \in \mathbb{N}$ such that for every $n \geq N(c)$, $\|c k f_n\|_{\mathcal{M}_{\varphi,w}} \leq 1$. By Proposition \ref{lem:basiclux}, this implies that $Q_{\varphi,w}(c k f_n) \leq 1$ for every $n \geq N(c)$. Now by the convexity of $\varphi$, we see that $c P_{\varphi,w}(k f_n) \leq c Q_{\varphi,w}(k f_n) \leq Q_{\varphi,w}(c k f_n) \leq 1$, and so $P_{\varphi,w}(k f_n) \leq Q_{\varphi,w}(k f_n) \leq \frac{1}{c}$ for every $n \geq N(c)$. Since $c>\max\{1, \frac{1}{k}\}$ is arbitrary, we obtain $\lim_{n \rightarrow \infty} P_{\varphi,w}(k f_n) = \lim_{n \rightarrow \infty} Q_{\varphi,w}(k f_n) = 0$ for every $k >0$.
	
	To show the converse, suppose that for every $k >0$, $\lim_{n \rightarrow \infty} P_{\varphi,w}(k f_n) = \lim_{n \rightarrow \infty} Q_{\varphi,w}(k f_n) = 0$. Then for every $c > \max\{1, \frac{1}{k}\}$, there exists  $N(c)$ such that for every $n \geq N(c)$, $Q_{\varphi,w}(k f_n) \leq \frac{1}{c}$. By the convexity of $Q_{\varphi,w}$ and the definition of $\|\cdot\|_{\mathcal{M}_{\varphi,w}}$, this implies that $\|f_n\|_{\mathcal{M}_{\varphi,w}} \leq \frac{1}{ck}$ for every $n \geq N(c)$. Since $c > \max\{1, \frac{1}{k}\}$ is arbitrary, we have $\lim_{n \rightarrow \infty} \|f_n\|_{\mathcal{M}_{\varphi,w}} = 0$. 
	
	The sequence case can be proven by the similar argument, so we omit the proof.       
\end{proof}  

\begin{Proposition}\label{lem:pbounded}
	Let $f \in L_0$ be a bounded function with support of finite measure. Then $P_{\varphi,w}(k f) \leq Q_{\varphi,w}(k f) < \infty$ for every $k >0$. The analogous statement also holds for a sequence $x \in \ell_0$ with finite entries. 
\end{Proposition}

\begin{proof}
	Let $f \in L_0$ be a bounded function with support of finite measure, that is $|f(t)| \le c$ a.e. on $I$ for some $c>0$, and let $k >0$. Since $f$ is equimeasurable to $f^*$, we see that $f^* \leq c$ a.e. on $(0, m(\supp{f^*}))$ where both $c$ and $m(\supp{f^*})$ are finite. By Proposition \ref{prop:Halmono}, we have $\varphi\left(\frac{k(f^*)^0}{w}\right)w \leq \varphi\left(\frac{k c(\chi_{(0, m(\supp{f^*}))})^0}{w}\right)w$ a.e. Hence, we have from (\ref{eq:charlevel}) that
	\begin{eqnarray*}
	P_{\varphi,w}(kf) \leq Q_{\varphi,w}(kf) = \int_I\varphi\left(\frac{k(f^*)^0}{w}\right)w &\leq& \int_I\varphi\left(\frac{k c(\chi_{(0, m(\supp{f^*}))})^0}{w}\right)w\\
	&=& \varphi\left(\frac{k c (m(\supp{f^*}))}{W(m(\supp{f^*}))}\right)W(m(\supp{f^*})) < \infty,
	\end{eqnarray*}
	and so $P_{\varphi,w}(k f) \leq Q_{\varphi,w}(k f) < \infty$ for every $k >0$. 
\end{proof}

We recall the following auxillary lemma that will be useful  later.

\begin{Lemma}\cite[Lemma 6.5]{FLM}\label{FLM}
	If $0 \leq f_n \leq f = f^* \in \mathcal{M}_{\varphi,w}$ such that $f_n = f_n^*$ for every $n \in \mathbb{N}$, $m(\supp{f_n}) < \infty$, and $f_n \rightarrow 0$ a.e., then $f_n^0 \rightarrow 0$ a.e. The same result holds for the sequence case.
\end{Lemma}

We are ready to provide an explicit description of the order-continuous subspace of $\mathcal{M}_{\varphi,w}$. 

\begin{Theorem}\label{Order}
	For the function space $\mathcal{M}_{\varphi,w}$,  
	\begin{eqnarray*}
	(\mathcal{M}_{\varphi,w})_a = (\mathcal{M}_{\varphi,w})_b &=& \{ f \in L_0 : P_{\varphi,w}(k f) < \infty \,\,\, \text{for all} \,\,\, k>0\}\\ 
	&=& \{ f \in L_0 : Q_{\varphi,w}(k f) < \infty \,\,\, \text{for all} \,\,\, k>0\}. 
	\end{eqnarray*}
	For the sequence space $\mathfrak{m}_{\varphi,w}$, 
	\begin{eqnarray*}
		(\mathfrak{m}_{\varphi,w})_a = (\mathfrak{m}_{\varphi,w})_b &=& \{ x \in \ell_0 : \mathfrak{p}_{\varphi,w}(k x) < \infty \,\,\, \text{for all} \,\,\, k >0\}\\
		&=& \{ x \in \ell_0 : \mathfrak{q}_{\varphi,w}(k x) < \infty \,\,\, \text{for all} \,\,\, k >0\}.
	\end{eqnarray*}
\end{Theorem}

\begin{proof}
	We already have $(\mathcal{M}_{\varphi,w})_a = (\mathcal{M}_{\varphi,w})_b$ from Lemma \ref{lem:Mfund} and \cite[Theorem 2.5.5]{BS}, so we only have to show the second and third equalities of the statement. Let $f \in  (\mathcal{M}_{\varphi,w})_b$. Then there exists a sequence $(f_n)_{n=1}^{\infty}$ of simple functions with supports of finite measure such that $\|f- f_n\|_{\mathcal{M}_{\varphi,w}} \rightarrow 0$. From Proposition \ref{modnormeq}, we see that $Q_{\varphi,w}(2k(f - f_n)) \rightarrow 0$ for every $k >0$. Hence for given $\epsilon >0$, there exists $n\in \mathbb{N}$ such that $Q_{\varphi,w}(2k(f - f_n)) < \epsilon$. By the convexity of $Q_{\varphi,w}$ and Proposition \ref{lem:pbounded}, we obtain
	\begin{eqnarray*}
		P_{\varphi,w}(k f) \leq Q_{\varphi,w}(k f) = Q_{\varphi,w}\left(\frac{2k f_n + 2k (f -f_n)}{2}\right) &\leq& \frac{1}{2} Q_{\varphi,w}(2k f_n) + \frac{1}{2} Q_{\varphi,w}(2k (f-f_n))\\
		&<& \frac{1}{2}Q_{\varphi,w}(2k f_n)  + \epsilon < \infty.   
	\end{eqnarray*}
	Thus, 
	\begin{eqnarray*}
	(\mathcal{M}_{\varphi,w})_a = (\mathcal{M}_{\varphi,w})_b &\subset& \{ f \in L_0 : Q_{\varphi,w}(k f) < \infty \,\,\, \text{for all} \,\,\, k >0\}\\
	&\subset& \{ f \in L_0 : P_{\varphi,w}(k f) < \infty \,\,\, \text{for all} \,\,\, k >0\}.
	\end{eqnarray*}
	
	Now, suppose $f \in L_0$ such that $Q_{\varphi,w}(k f) < \infty$ for every $k >0$. Define $f_n = f \chi_{\{\frac{1}{n} < |f| < n\}}$. From Proposition \ref{prop:distfin}, $m\{\frac{1}{n} < |f| < n\} \leq m \{ |f| > \frac{1}{n}\} < \infty$, so $f_n \in (\mathcal{M}_{\varphi,w})_b$ for all $n \in \mathbb{N}$. 
	Since $|f_n| \leq |f|$ a.e. and $f - f_n \rightarrow 0$ a.e., we have $(f - f_n)^* \rightarrow 0$ a.e. by Theorem \ref{convto0dec}.  Moreover,  $(f- f_n)^* \leq f^*$ a.e. \cite[Proposition 1.7]{BS} and so $((f - f_n)^*)^0 \rightarrow 0$ a.e. and $((f - f_n)^*)^0 \leq (f^*)^0$ a.e. in view of Lemma \ref{FLM} and Proposition \ref{prop:Halmono}. Hence for every $k >0$, $\varphi\left(\frac{k((f - f_n)^*)^0}{w}\right)w \rightarrow 0$ a.e. and $\varphi\left(\frac{k((f - f_n)^*)^0}{w}\right)w \leq \varphi\left(\frac{k((f^*)^0}{w}\right)w$ a.e. By the Lebesgue Dominated Convergence Theorem, we have $Q_{\varphi,w}(k(f-f_n)) = \int_I \varphi\left(\frac{k((f - f_n)^*)^0}{w}\right)w \rightarrow 0$. Then $\|f - f_n\|_{\mathcal{M}_{\varphi,w}} \rightarrow 0$ by Proposition \ref{modnormeq}, and this shows that $\{ f \in L_0 : Q_{\varphi,w}(k f) < \infty \,\,\, \text{for all} \,\,\, k >0\} \subset (\mathcal{M}_{\varphi,w})_a = (\mathcal{M}_{\varphi,w})_b$.
	
	For $f \in L_0$ such that $P_{\varphi,w}(kf) < \infty$ for every $k > 0$, we define the functions $f_n \in (\mathcal{M}_{\varphi,w})_b$ as before. Then for every $\epsilon > 0$, there exists $v_{\epsilon}\prec w$ such that 
		\[
		\int_{I}\varphi\left(\frac{kf^*}{v_{\epsilon}}\right)v_{\epsilon} \leq P_{\varphi,w}(kf) + \epsilon < \infty.
		\]
	From the fact that $|f - f_n| \leq |f|$ a.e., we have $(f - f_n)^* \leq f^*$ a.e. \cite[Proposition 1.7]{BS}, and so $\int_{I}\varphi\left(\frac{k(f - f_n)^*}{v_{\epsilon}}\right)v_{\epsilon} \leq \int_{I}\varphi\left(\frac{kf^*}{v_{\epsilon}}\right)v_{\epsilon} < \infty$. Since $\varphi\left(\frac{k(f - f_n)^*}{v_{\epsilon}}\right)v_{\epsilon} \rightarrow 0$ a.e. as $n \rightarrow \infty$, by the Lebesgue Dominated Convergence Theorem, $\int_I \varphi\left(\frac{k(f - f_n)^*}{v_{\epsilon}}\right)v_{\epsilon} \rightarrow 0$ as $n \rightarrow \infty$. This implies that
	\[
	P_{\varphi,w}(k(f - f_n)) \leq \int_{I}\varphi\left(\frac{k(f - f_n)^*}{v_{\epsilon}}\right)v_{\epsilon}\rightarrow 0,
	\]
	as $n \rightarrow \infty$. Thus by Proposition \ref{modnormeq}, $\|f - f_n\|_{\mathcal{M}_{\varphi,w}} \rightarrow 0$, and so we have $\{ f \in L_0 : P_{\varphi,w}(k f) < \infty \,\,\, \text{for all} \,\,\, k >0\} \subset (\mathcal{M}_{\varphi,w})_a = (\mathcal{M}_{\varphi,w})_b$. 
	
	For the sequence spaces, we always have $(\mathfrak{m}_{\varphi,w})_a =(\mathfrak{m}_{\varphi,w})_b$ \cite[Theorem 2.5.4]{BS}. By replacing $P_{\varphi,w}$ with $\mathfrak{p}_{\varphi,w}$ and $Q_{\varphi,w}$ with $\mathfrak{q}_{\varphi,w}$, we can prove the second equality by the similar argument.   
\end{proof}

The following observation will be useful.

\begin{Corollary}\label{cor:normone}
	If $f \in (\mathcal{M}_{\varphi,w})_a$ such that $\|f\|_{\mathcal{M}_{\varphi,w}} = 1$, then $P_{\varphi,w}(f) = Q_{\varphi,w}(f) = 1$.  The analogous statement holds for the sequence $x \in (\mathfrak{m}_{\varphi,w})_a$ with $\|x\|_{\mathfrak{m}_{\varphi,w}} = 1$.
\end{Corollary}

\begin{proof}
	In view of Theorem \ref{Order}, for $f\in (\mathcal{M}_{\varphi,w})_a$,  the mapping $k \mapsto P_{\varphi,w}(kf)$ is continuous on $[0, \infty)$, and {\bf $\lim_{k\to\infty} P_{\varphi,w}(kf) = \infty$}. Hence, there exists $k_0$ such that $P_{\varphi,w}(k_0f) = 1$. By the definition of the Luxemburg norm of $\mathcal{M}_{\varphi,w}$, we must have $k_0 = \|f\|_{\mathcal{M}_{\varphi,w}} = 1$. By considering the continuous mapping $k \mapsto Q_{\varphi,w}(kf)$ on $[0, \infty)$ with $\lim_{k \to \infty}Q_{\varphi,w}(kf) = \infty$, we can also show that $Q_{\varphi,w}(f) = 1$. 
\end{proof}

Before presenting the main results of this section, we need the following auxiliary results.
	   
\begin{Lemma}\label{cor:notdelta2}
	If $\varphi$ does not satisfy the appropriate $\Delta_2$-condition, there exists $f \in \mathcal{M}_{\varphi,w}$ such that $Q_{\varphi,w}(f) \leq 1$ but $ Q_{\varphi,w}(rf) = \infty$ for $r > 1$. The similar result holds for the space $\mathfrak{m}_{\varphi,w}$.
\end{Lemma}

\begin{proof}
		Since the proof for the infinite Lebesgue measure space is similar to the finite Lebesgue measure space case, we only consider when $\gamma < \infty$. In this case, we have $\int_0^{\gamma} w = W(\gamma) < \infty$. Assume that $\varphi$ does not satisfy the $\Delta_2^{\infty}$-condition. Then by Lemma \ref{le:C}, there exists an increasing sequence of real numbers $(u_n)_{n=1}^{\infty}$ with $u_n \uparrow \infty$ such that 
	\begin{equation}\label{eq:eqvD2}
		\varphi\left(\left(1 + \frac{1}{n}\right)u_n\right) > 2^n \varphi(u_n) \ \ \text{for every} \ \  n \in \mathbb{N}.
	\end{equation}
	
	Choose $u_1$ that satisfies $\frac{1}{\varphi(u_1)} < W(\gamma)$. Since $\sum_{n=1}^{\infty} \frac{1}{2^n \varphi(u_n)} \leq \frac{1}{\varphi(u_1)} < W(\gamma)$, there exists $t_0 \in (0, \gamma)$ such that $W(t_0) = \sum_{n=1}^{\infty} \frac{1}{2^n \varphi(u_n)}$. Hence we can choose $t_n \downarrow 0$ such that $\frac{1}{2^n\varphi(u_n)} = \int_{t_n}^{t_{n-1}} w$ for every $n \in \mathbb{N}$. Let $(t_{n_i})_{i=1}^{\infty}$ be a subsequence of $(t_n)_{n=1}^{\infty}$ such that $\sum_{i=0}^{\infty} t_{n_i} < \gamma$, where $t_{n_0} = t_0$.  For $k \in \mathbb{N}$, define
	\[
	f_k = \sum_{n = n_k+1}^{\infty} u_n w \chi_{[t_n + \sum_{i=0}^{k-1}t_{n_i}, t_{n-1} + \sum_{i=0}^{k-1}t_{n_i})}.
	\]
	Notice that $\supp f_k \subset [\sum_{i=0}^{k-1}t_{n_i}, \sum_{i=0}^{k}t_{n_i})$ and that $\supp f_k$'s are pairwise disjoint. 
	
	Due to our choice of $u_n$ and $w$ being decreasing, we see that $f_k^* = \sum_{n = n_k+1}^{\infty} u_n w \chi_{[t_n, t_{n-1})}$ and  $\frac{f_k^*}{w} = \sum_{n = n_k+1}^{\infty} u_n \chi_{[t_n, t_{n-1})}$ are decreasing functions, so $f_k^* = (f_k^*)^0$ by Proposition \ref{prop:declevel}. Hence we obtain 
	\[
	Q_{\varphi,w}(f_k) = \int_I \varphi \left(\frac{(f_k^*)^0}{w}\right)w = \sum_{n = n_k + 1}^{\infty} \varphi(u_n) \int_{t_n}^{t_{n-1}} w = \frac{1}{2^{n_k}}\leq \frac{1}{2^k}.
	\] 
	
	Now, let $s > 1$. Then there exists $j_0$ such that for all $n > j_0$, $1 + \frac{1}{n} < s$. Let  $k\in\mathbb{N}$ and $N = \max\{j_0, n_k +1\}$. For every $n \geq N$, by (\ref{eq:eqvD2}), we have
	
	\[
	Q_{\varphi,w}(sf_k) = \sum_{n = n_k + 1}^{\infty} \varphi(s u_n) \int_{t_n}^{t_{n-1}} w \geq \sum_{n = N}^{\infty} \varphi\left(\left(1 + \frac{1}{n}\right) u_n\right) \int_{t_n}^{t_{n-1}} w > \sum_{n = N}^{\infty} \frac{2^n\varphi(u_n)}{2^n\varphi(u_n)}= \infty.    
	\]
	Hence $\|f_k\|_{\mathcal{M}_{\varphi,w}} = 1$.
\end{proof}

\begin{Lemma}\label{cor:notdelta2-P}
	If $\varphi$ does not satisfy the appropriate $\Delta_2$-condition, there exists $f \in \mathcal{M}_{\varphi,w}$ such that $P_{\varphi,w}(f) \leq 1$ but $ P_{\varphi,w}(rf) = \infty$ for $r > 1$. The similar result holds for the space $\mathfrak{m}_{\varphi,w}$.
\end{Lemma}

\begin{proof}
		Since the proof for the infinite Lebesgue measure space is similar to the finite Lebesgue measure space case, we only consider when $\gamma < \infty$. Fix a decreasing function $v \prec w$. In this case, we have $\int_0^{\gamma} v = V(\gamma) \leq W(\gamma) < \infty$. Assume that $\varphi$ does not satisfy the $\Delta_2^{\infty}$-condition. Then by Lemma \ref{le:C}, there exists an increasing sequence of real numbers $(u_n)_{n=1}^{\infty}$ with $u_n \uparrow \infty$ such that 
	\begin{equation}\label{eq:eqvD2-2}
		\varphi\left(\left(1 + \frac{1}{n}\right)u_n\right) > 2^n \varphi(u_n) \ \ \text{for every} \ \  n \in \mathbb{N}.
	\end{equation}
	
	Choose $u_1$ that satisfies $\frac{1}{\varphi(u_1)} < V(\gamma)$. Since $\sum_{n=1}^{\infty} \frac{1}{2^n \varphi(u_n)} \leq \frac{1}{\varphi(u_1)} < V(\gamma)$, there exists $t_0 \in (0, \gamma)$ such that $V(t_0) = \sum_{n=1}^{\infty} \frac{1}{2^n \varphi(u_n)}$. Hence we can choose $t_n \downarrow 0$ such that $\frac{1}{2^n\varphi(u_n)} = \int_{t_n}^{t_{n-1}} v$ for every $n \in \mathbb{N}$. Let $(t_{n_i})_{i=1}^{\infty}$ be a subsequence of $(t_n)_{n=1}^{\infty}$ such that $\sum_{i=0}^{\infty} t_{n_i} < \gamma$, where $t_{n_0} = t_0$.  For $k \in \mathbb{N}$, define
	\[
	f_k = \sum_{n = n_k+1}^{\infty} u_n v \chi_{[t_n + \sum_{i=0}^{k-1}t_{n_i}, t_{n-1} + \sum_{i=0}^{k-1}t_{n_i})}.
	\]
	Notice that $\supp f_k \subset [\sum_{i=0}^{k-1}t_{n_i}, \sum_{i=0}^{k}t_{n_i})$ and that $\supp f_k$'s are pairwise disjoint. 
	
	Due to our choice of $u_n$ and $v$ being decreasing, we see that $f_k^* = \sum_{n = n_k+1}^{\infty} u_n v \chi_{[t_n, t_{n-1})}$ and $\frac{f_k^*}{v} = \sum_{n = n_k+1}^{\infty} u_n \chi_{[t_n, t_{n-1})}$. Hence we obtain 
	\[
	P_{\varphi,w}(f_k) \leq \int_I \varphi \left(\frac{f_k^*}{v}\right)v = \sum_{n = n_k + 1}^{\infty} \varphi(u_n) \int_{t_n}^{t_{n-1}} v = \frac{1}{2^{n_k}}\leq \frac{1}{2^k}.
	\] 
	
	Now, let $s > 1$. Then there exists $j_0$ such that for all $n > j_0$, $1 + \frac{1}{n} < s$. Let  $k\in\mathbb{N}$ and $N = \max\{j_0, n_k +1\}$. For $n \geq N$, by (\ref{eq:eqvD2-2}), we have
	
	\[
	\sum_{n = n_k + 1}^{\infty} \varphi(s u_n) \int_{t_n}^{t_{n-1}} v \geq \sum_{n = N}^{\infty} \varphi\left(\left(1 + \frac{1}{n}\right) u_n\right) \int_{t_n}^{t_{n-1}} v > \sum_{n = N}^{\infty} \frac{2^n\varphi(u_n)}{2^n\varphi(u_n)}= \infty.    
	\]
	Since this is true for every decreasing $v \prec w$, $P_{\varphi,w}(sf_k) = \infty$ for every $s > 1$. 
	Therefore, $\|f_k\|_{\mathcal{M}_{\varphi,w}} = 1$.
\end{proof}

\begin{Theorem}\label{notD}
	If $\varphi$ does not  satisfy the appropriate $\Delta_2$-condition, there exists a sequence of non-negative functions $f_k \in \mathcal{M}_{\varphi,w}$ with pairwise disjoint supports such that $\|f_k\|_{\mathcal{M}_{\varphi,w}} = 1$  and $\|\sum_{k=1}^{\infty} f_k\|_{\mathcal{M}_{\varphi,w}} =1$ for every $k\in \mathbb{N}$. The similar result holds for the space $\mathfrak{m}_{\varphi,w}$.  
\end{Theorem}

\begin{proof}	
	Let $(f_k)_{k = 1}^{\infty}$ be the sequence of pairwise disjoint functions constructed in Lemma \ref{cor:notdelta2}. In order to show $\|\sum_{k=1}^{\infty}f_k\|_{\mathcal{M}_{\varphi,w}} = 1$, observe that $1= \|f_k\|_{\mathcal{M}_{\varphi,w}} \leq \|\sum_{k=1}^{\infty} f_k\|_{\mathcal{M}_{\varphi,w}}$. 
	Let $r>1$. Since $Q_{\varphi,w}(r f_k) = \infty$ for all $k$, we have
	\[
	Q_{\varphi,w}\left(r \sum_{k=1}^{\infty}f_k\right) \geq Q_{\varphi,w}(r f_k) = \infty.
	\]
	Thus, $\|\sum_{k=1}^{\infty}f_k\|_{\mathcal{M}_{\varphi,w}} = 1$. The same statement can be proven by using the sequence of $(f_k)_{k = 1}^{\infty}$ constructed in Lemma \ref{cor:notdelta2-P} if one wishes to prove it in terms of the modular $P_{\varphi,w}$.
	
	Since the analogous result for the space $\mathfrak{m}_{\varphi,w}$ can be shown with the similar method, we omit the proof.
\end{proof}

Now, we are ready to prove the main result of this section. Even though the implication (ii) $\implies$ (i) of the next theorem was already shown in \cite[Theorem 6.7]{FLM}, we prove this with a different method. More importantly, we can see that the reverse implication also holds. 

\begin{Theorem}\label{Sep}
	The following statements are equivalent:
	\begin{enumerate}[\rm(i)]
		\item $\mathcal{M}_{\varphi,w} = \{ f \in L_0 : Q_{\varphi,w}(k f) < \infty \,\,\, \text{for all} \,\,\, k >0\} = \{ f \in L_0 : P_{\varphi,w}(k f) < \infty \,\,\, \text{for all} \,\,\, k >0\}$.
		\item $\varphi$ satisfies the appropriate $\Delta_2$-condition. 
	\end{enumerate}
\end{Theorem}

\begin{proof}
	Assume to the contrary that $\varphi$ does not satisfy the appropriate $\Delta_2$-condition. Then by Lemma \ref{cor:notdelta2-P}, there exists $f \in \mathcal{M}_{\varphi,w}$ such that $P_{\varphi,w}(f) \leq 1$ and $P_{\varphi,w}(2f) = \infty$. Hence $f \notin \{ f \in L_0 : P_{\varphi,w}(k f) < \infty \,\,\, \text{for all} \,\,\, k >0\}$. Moreover, since $\{ f \in L_0 : Q_{\varphi,w}(k f) < \infty \,\,\, \text{for all} \,\,\, k >0\} \subset \{ f \in L_0 : P_{\varphi,w}(k f) < \infty \,\,\, \text{for all} \,\,\, k >0\}$ in general, we also see that $f \notin \{ f \in L_0 : Q_{\varphi,w}(k f) < \infty \,\,\, \text{for all} \,\,\, k >0\}$. This shows (i) $\Rightarrow$ (ii).
	
	To prove (ii) $\implies$ (i), we only present the proof in regards to finite measure spaces because the proof for infinite measure spaces can be modified  easily. Suppose that $\varphi$ satisfies the $\Delta_2^{\infty}$-condition. Then there exist $K > 0$ and $u_0 \geq 0$ such that $\varphi(2u) \leq K\varphi(u)$ for every $u \geq u_0$. 
	
	Let $f \in \mathcal{M}_{\varphi,w}$. Then $Q_{\varphi,w}(k f) < \infty$ for some $k > 0$. It is enough to show that $P_{\varphi,w}(2k f) < \infty$. Define a set $A = \left\{ \frac{(f^*)^0}{w} \geq \frac{u_0}{k}\right\}$. There exists $ 0<c\le\gamma$ such that $A = (0, c) \subset I$ because $\frac{(f^*)^0}{w}$ is decreasing. From the fact that $\frac{(f^*)^0}{w} < \frac{u_0}{k}$ on $(c, \gamma)$ and the $\Delta_2^{\infty}$-condition of $\varphi$, we have
	\begin{eqnarray*}
		Q_{\varphi,w}(2k f) = \int_I \varphi\left(\frac{2k (f^*)^0}{w}\right) w &=& \int_0^c \varphi\left(\frac{2k (f^*)^0}{w}\right) w + \int_c^{\gamma} \varphi\left(\frac{2k (f^*)^0}{w}\right) w\\
		&\leq& K\int_0^c \varphi\left(\frac{k (f^*)^0}{w}\right) w + \varphi(2u_0)(W(\gamma) - W(c))\\
		&\leq& K Q_{\varphi,w}(k f) + \varphi(2u_0)(W(\gamma) - W(c)) < \infty.
	\end{eqnarray*}
 	Hence, $P_{\varphi,w}(2kf) < \infty$.   
\end{proof}    

We have the analogous result for the sequence spaces.
\begin{Theorem}\label{seqSep}
	The following statements are equivalent:
	\begin{enumerate}[\rm(i)]
		\item $\mathfrak{m}_{\varphi,w}= \{ x \in \ell_0 : \mathfrak{p}_{\varphi,w}(\lambda x) < \infty \,\,\, \text{for all} \,\,\, \lambda >0\} = \{ x \in \ell_0 : \mathfrak{q}_{\varphi,w}(\lambda x) < \infty \,\,\, \text{for all} \,\,\, \lambda >0\}$;
		\item $\varphi$ satisfies $\Delta_2^0$-condition.
	\end{enumerate}
\end{Theorem}

Hence we have the following consequence of Theorem \ref{Sep} and Theorem \ref{seqSep}.

\begin{Corollary}\label{cor:sep}
	The function space $\mathcal{M}_{\varphi,w}$ is separable if and only if $\varphi$ satisfies the appropriate $\Delta_2$-condition. The sequence space $\mathfrak{m}_{\varphi,w}$ is separable if and only if $\varphi$ satisfies the $\Delta_2^0$-condition. 
\end{Corollary}

\begin{proof}
	From Corollary \ref{Sep}, $\varphi$ satisfies the appropriate $\Delta_2$-condition if and only if $\mathcal{M}_{\varphi,w} = \{ f \in L_0 : P_{\varphi,w}(\lambda f) < \infty \,\,\, \text{for all} \,\,\, \lambda >0\} =  \{ f \in L_0 : Q_{\varphi,w}(\lambda f) < \infty \,\,\, \text{for all} \,\,\, \lambda >0\}$. In view of Theorem \ref{Order}, $\mathcal{M}_{\varphi,w} = (\mathcal{M}_{\varphi,w})_a = (\mathcal{M}_{\varphi,w})_b$, so the space $\mathcal{M}_{\varphi,w}$ is order-continuous. Since the Lebesgue measure $m$ is separable, the space $\mathcal{M}_{\varphi,w}$ is also separable \cite[Theorem 1.5.5]{BS}.
	
	For the space $\mathfrak{m}_{\varphi,w}$, from Theorem \ref{seqSep}, $\varphi$ satisfies the $\Delta_2^0$-condition if and only if $\mathfrak{m}_{\varphi,w} = \{ x\in \ell_0 : \mathfrak{p}_{\varphi,w}(\lambda x) < \infty \,\,\, \text{for all} \,\,\, \lambda >0\} = \{ x\in \ell_0 : \mathfrak{q}_{\varphi,w}(\lambda x) < \infty \,\,\, \text{for all} \,\,\, \lambda >0\}$. By Theorem \ref{Order}, $\mathfrak{m}_{\varphi,w} = (\mathfrak{m}_{\varphi,w})_a = (\mathfrak{m}_{\varphi,w})_b$ is separable.
\end{proof}

A Banach function lattice $X$ is a {\it $KB$-space} if the space is order-continuous and has the Fatou property \cite[Definition 5.1, pg 84]{Sch}. In addition, we mention that a Banach function lattice $X$ is a $KB$-space if and only if $X$ does not have an isomorphic copy of $c_0$ \cite[Theorem 4.60]{AB2}.  

\begin{Theorem}\label{th:copy}
	The following statements are equivalent:
	\begin{enumerate}[\rm(i)]
		\item $\varphi$ satisfies the $\Delta_2$-condition when $\gamma = \infty$ or $\varphi$ satisfies the $\Delta_2^{\infty}$-condition when $\gamma < \infty$ (resp. $\varphi$ satisfies the $\Delta_2^0$-condition);
		\item $\mathcal{M}_{\varphi,w}$ (resp. $\mathfrak{m}_{\varphi,w}$) does not have an isometric copy of $l^{\infty}$;
		\item $\mathcal{M}_{\varphi,w}$ (resp. $\mathfrak{m}_{\varphi,w}$) does not have an isometric copy of $c_0$;
		\item $\mathcal{M}_{\varphi,w}$ (resp. $\mathfrak{m}_{\varphi,w}$) does not have an isomorphic copy of $c_0$;
		\item $\mathcal{M}_{\varphi,w}$ (resp. $\mathfrak{m}_{\varphi,w}$) is a $KB$-space;
		\item $\mathcal{M}_{\varphi,w}$ (resp. $\mathfrak{m}_{\varphi,w}$) is separable. 
	\end{enumerate}
\end{Theorem}

\begin{proof}
	(iv) $\implies$ (iii) $\implies$ (ii) is clear. Since the Lebesgue measure $m$ is separable, the order-continuity of the space is equivalent to the separability of that space \cite[Theorem 1.5.5]{BS}. Moreover, from the fact that the spaces $\mathcal{M}_{\varphi,w}$ and $\mathfrak{m}_{\varphi,w}$ have the Fatou property \cite[Theorem 4.7]{KR}, we obtain (v) $\iff$ (vi).  The implication (iv) $\iff$ (v) is from \cite[Theorem 4.60]{AB2}. The implication (i) $\iff$ (vi) was shown in Corollary \ref{cor:sep}. 
	
	Hence, we only have to show (ii) $\implies$ (i). Let $\varphi$ be an Orlicz function which does not satisfy the appropriate $\Delta_2$-condition. In view of Theorem \ref{notD}, let $(f_k)_{k=1}^{\infty}$ be a sequence of non-negative functions in $\mathcal{M}_{\varphi,w}$ with pairwise disjoint supports such that $\|f_k\|_{\mathcal{M}_{\varphi,w}} =1$ and $\|\sum_{k=1}^{\infty}f_k\|_{\mathcal{M}_{\varphi,w}} = 1$ for every $k\in \mathbb{N}$. Define an operator $T : l^{\infty} \rightarrow \mathcal{M}_{\varphi,w}$ such that $Tx = \sum_{k=1}^{\infty} |x(k)| f_k$, where $x = (x(k))_{k=1}^{\infty} \in \ell_{\infty}$. Then we get  
	\[
	\|Tx\|_{\mathcal{M}_{\varphi,w}} = \left\|\sum_{k=1}^{\infty} |x(k)| f_k\right\|_{\mathcal{M}_{\varphi,w}} \leq \|x\|_{\infty}\left\|\sum_{k=1}^{\infty} f_k\right\|_{\mathcal{M}_{\varphi,w}} = \|x\|_{\infty}.
	\]
	
	Observe that for any $0<\lambda<1$, there exists $|x(k_0)|$ such that $\frac{|x(k_0)|}{\lambda \|x\|_{\infty}} > 1$. Hence by Lemma \ref{cor:notdelta2}, we have
	\[
	Q_{\varphi,w} \left(\frac{\sum_{k=1}^{\infty}|x(k)| f_k}{\lambda \|x\|_{\infty}}\right) > Q_{\varphi, w}\left(\frac{|x(k_0)| f_{k_0}}{\lambda \|x\|_{\infty}}\right) = \infty,
	\]
	which shows that $\|Tx\|_{\mathcal{M}_{\varphi,w}} = \|x\|_{\infty}$. Thus, the space $\mathcal{M}_{\varphi,w}$ contains an isometric copy of $l^\infty$ if $\varphi$ does not satisfy the appropriate $\Delta_2$-condition. The proof is finished.
	
	To show (ii) $\implies$ (i) for the space $\mathfrak{m}_{\varphi,w}$, the similar argument from the function case is applied in view of Theorem \ref{notD} and Lemma \ref{cor:notdelta2}.
\end{proof}

We finish this section with an observation. A Banach space $X$ is said to have the {\it Radon-Nikod\'ym property with respect to $(\Omega, \Sigma, \mu)$} if for every $\mu$-continuous vector measure $G: \Sigma \rightarrow X$ of bounded variation, there exists a Bochner integrable function $g:\Omega \rightarrow X$ such that $G(A) = \int_A g d\mu$ for every $A \in \Sigma$. If this holds for every finite measure space $(\Omega, \Sigma, \mu)$, then we say the space $X$ has the {\it Radon-Nikod\'ym property} (RNP). 

It is well-known that every separable dual space has the RNP, which in turn implies that the space has slices of arbitrarily small diameter (c.f. \cite{DU}). This property is in the opposite spectrum of Banach spaces with the Daugavet property and the diameter two properties, which have been active research areas in geometry of Banach spaces. In view of Theorem \ref{cor:sep}, we obtain the following consequence.

\begin{Theorem}\label{th:RNPM}
	Let $\varphi$ be an Orlicz N-function at infinity that satisfies the appropriate $\Delta_2$-condition. Then the space $\mathcal{M}_{\varphi,w}$ has the RNP. As a consequence, the unit ball of ${\mathcal{M}_{\varphi,w}}$ has slices of arbitrarily small diameters. 
\end{Theorem}
 
 \begin{proof}
 	We claim that $\mathcal{M}_{\varphi,w}$ is a separable dual space. Since $\varphi$ is an $N$-function at infinity, the complementary function $\varphi_*$ is finite \cite{KLT4}, and so the inverse function $\varphi_*^{-1}$ is well-defined on $\mathbb{R}^{+}$. It is well-known that the fundamental function of $\Lambda_{\varphi_*, w}$ equipped with the Luxemberg norm is
 	\[
 	\phi_{\Lambda_{\varphi_*, w}}(t) = \|\chi_{[0, t]}\|_{\Lambda_{\varphi_*, w}} = \frac{1}{\varphi_*^{-1}(1/W(t))},
 	\]
 	where $W(t) = \int_0^t w(s)ds$, $t \in (0, \gamma)$ \cite{KR, KLT3}. Since the Luxemburg norm and the Orlicz norm are equivalent for Orlicz-Lorentz spaces (see pg 4), we have 
 	\[
 	\lim_{t \rightarrow 0+}\phi_{\Lambda_{\varphi_*,w}^0}(t) \leq 2\lim_{t \rightarrow 0+}\phi_{\Lambda_{\varphi_*,w}}(t).
 	\] 
 	
 	Hence, $\lim_{t \rightarrow 0+}\phi_{\Lambda_{\varphi_*,w}^0}(t) = 0$, and so $(\Lambda_{\varphi_*,w}^0)_a = (\Lambda_{\varphi_*,w}^0)_b$ by \cite[Theorem 2.5.5]{BS}. Moreover, from the fact that $\varphi = \varphi_{**}$, we have $ ((\Lambda_{\varphi_*,w}^0)_a)^* = (\Lambda_{\varphi_*,w}^0)' =  \mathcal{M}_{\varphi,w}$ because $\mathcal{M}_{\varphi,w}$ has the Fatou property. Therefore, the space $\mathcal{M}_{\varphi,w}$ is a dual space. The separability is a direct consequence of Theorem \ref{cor:sep}. Thus the conclusion is proved.
 \end{proof}

\begin{Remark}
	For an arbitrary Orlicz function, in particular not being an $N$-function at infinity, the analogous statement in Theorem \ref{th:RNPM} does not hold. For instance, if $\varphi(u) = ku$ for some $k > 0$, then the space $\mathcal{M}_{\varphi,w}$ becomes $L_1$. Since $L_1$ has the Daugavet property, the space does not have the RNP.
\end{Remark}

\section{Comparison between the spaces $\mathcal{M}_{\varphi,w}$}
In this section, we study the comparison between $\mathcal{M}_{\varphi,w}$ spaces given by different Orlicz functions $\varphi$. For this purpose, a certain order between Orlicz functions plays an important role. For two Orlicz functions $\varphi$ and $\psi$, we denote the order $\varphi \prec \psi$ (resp. $\varphi \prec_{\infty} \psi$)
if there exists $b > 0$
(resp. there exist $b > 0$ and $u_0 \geq 0$) such that $\varphi(u) \leq \psi(bu)$ for every $u \geq 0$ (resp. $u \geq u_0$) \cite[Definition 2.2.1]{RR}. Using this order relation, the comparison results on Orlicz-Lorentz spaces were studied in \cite{K, BRST}.
    
\begin{Theorem}\label{th:comparison}
	Let $\varphi_1$ and $\varphi_2$ be Orlicz functions. Then the following statements are equivalent:
	\begin{enumerate}[\rm(i)]
		\item $\varphi_1 \prec \varphi_2$ (resp. $\varphi_1 \prec_{\infty} \varphi_2)$ when $\gamma = \infty$ (resp. $\gamma < \infty$).  
		\item $\mathcal{M}_{\varphi_2,w} \hookrightarrow \mathcal{M}_{\varphi_1,w}$. 
	\end{enumerate}
	
\end{Theorem}

\begin{proof}
	We only show the proof of the statement when $\gamma < \infty$. Suppose that $\varphi_1 \prec_{\infty} \varphi_2$. Let $f \in \mathcal{M}_{\varphi_2,w}$. There exist $b > 0$ and $u_0 \ge 0$ such that $\varphi_1(u) \leq \varphi_2(bu)$ for every $u \geq u_0$. Define a set $E = \{t \in I : (f^*)^0(t) \leq u_0b\|f\|_{\mathcal{M}_{\varphi_2,w}}w(t)\}$. From the fact that $\frac{(f^*)^0}{w}$ is a decreasing function, the set $E$ is actually the interval $[\gamma - mE, \gamma]$. Let $M = \varphi_1(u_0)(W(\gamma) - W(\gamma - mE)) + 1$. Since $Q_{\varphi, w}$ is a convex modular and $Q_{\varphi_2,w}\left(\frac{f}{\|f\|_{\mathcal{M}_{\varphi_2,w}}}\right) \leq 1$, 
	
	\begin{eqnarray*}
		Q_{\varphi_1, w} \left(\frac{f}{Mb\|f\|_{\mathcal{M}_{\varphi_2,w}}}\right) &=& \int_I \varphi_1 \left(\frac{(f^*)^0}{Mb\|f\|_{\mathcal{M}_{\varphi_2,w}}w}\right)w\\
		&\leq& \frac{1}{M} \left(\int_{\gamma - mE}^{\gamma}\varphi_1 \left(\frac{(f^*)^0}{b\|f\|_{\varphi_2,w}w}\right)w + \int_I\varphi_1 \left(\frac{(f^*)^0}{b\|f\|_{\varphi_2,w}w}\right)w\right)\\
		&\leq& \frac{1}{M} \left(\varphi_1(u_0)(W(\gamma) - W(\gamma - mE))  + \int_I\varphi_2 \left(\frac{(f^*)^0}{\|f\|_{\varphi_2,w}w}\right)w\right)\\
		&\leq& \frac{1}{M} \cdot M = 1.
	\end{eqnarray*}	
	Thus, we obtain $\|f\|_{\mathcal{M}_{\varphi_2,w}} \leq b \|f\|_{\mathcal{M}_{\varphi_1,w}}$ by the definition of the Luxemburg norm $\|\cdot\|_{\mathcal{M}_{\varphi,w}}$, and so $\mathcal{M}_{\varphi_2,w} \subset \mathcal{M}_{\varphi_1,w}$. This shows (i) $\implies$ (ii). 
	
	For (ii) $\implies$ (i), assume to the contrary that $\varphi_1 \not\prec_{\infty} \varphi_2$. Then for every $\epsilon > 0$, there exists a sequence of real numbers $(a_n)_{n=1}^{\infty}\uparrow \infty$ such that $\varphi_1 (a_n) > \varphi_2(2^n n^2 a_n)$ for every $n \in \mathbb{N}$. We see from the convexity of the Orlicz function $\varphi_2$ that $\varphi_1(a_n) > \varphi_2(2^n n^2 a_n) \geq 2^n\varphi_2(n^2 a_n)$ for all $n \in \mathbb{N}$.
	
	Passing to a subsequence if necessary, we can choose the sequence $(a_n)_{n=1}^{\infty}$ such that 
	\[
	\sum_{n=1}^{\infty} \frac{1}{2^n \varphi_2(n^2 a_n)} < \int_{0}^{\gamma} w.
	\]       
	Let $t_0$ be a positive real number such that 
	\[
	\sum_{n=1}^{\infty} \frac{1}{2^n \varphi_2(n^2 a_n)}  = \int_0^{t_0}w.
	\]
	Then, there exists a sequence of real numbers $(t_n)_{n=1}^{\infty} \downarrow 0$ such that 
	\[
	\frac{1}{2^n \varphi_2(n^2 a_n)} = \int_{t_n}^{t_{n-1}}w.
	\]  
	Now, define
	\[
	f := \sum_{n=1}^{\infty} na_nw\chi_{[t_n, t_{n-1})}.
	\]
	By Proposition \ref{prop:declevel}, we see that $f = f^* = (f^*)^0$. Hence, 
	\[
	Q_{\varphi_2,w}(f) = \sum_{n=1}^{\infty} \varphi_2(na_n) \int_{t_n}^{t_{n-1}}w = \sum_{n=1}^{\infty} \frac{\varphi_2(na_n)}{2^n \varphi_2(n^2 a_n)} < \sum_{n=1}^{\infty}\frac{1}{2^n} = 1.
	\] 
	This shows that $f \in \mathcal{M}_{\varphi_2,w}$. On the other hand, for every $\epsilon > 0$ there exists $n_0 \in \mathbb{N}$ such that $n_0 \epsilon > 1$, so 
		\begin{eqnarray*}
			Q_{\varphi_1,w}(\epsilon f) \geq \sum_{n=n_0 + 1}^{\infty} \varphi_1 (\epsilon n a_n) \int_{t_n}^{t_{n-1}}w \geq \sum_{n=n_0 + 1}^{\infty} \varphi_1 \left(a_n\right) \int_{t_n}^{t_{n-1}}w \geq \sum_{n=n_0 + 1}^{\infty} \frac{2^n\varphi_2(n^2 a_n)}{2^n\varphi_2(n^2 a_n)} = \infty. 
	\end{eqnarray*}
	Therefore, $f \notin \mathcal{M}_{\varphi_1,w}$ and so $\mathcal{M}_{\varphi_2,w} \not\subset \mathcal{M}_{\varphi_1,w}$.
\end{proof}

\section{The Orlicz norm in $\mathcal{M}_{\varphi,w}$}\label{sec:OrliczM}
From Theorem \ref{th:KLR} and the fact that $\mathcal{M}_{\varphi,w}$ is r.i. \cite[Theorem 4.7]{KR}, we can also look at the Orlicz norm for $\mathcal{M}_{\varphi,w}$ as the K\"othe dual norm of $\Lambda_{\varphi_*,w}$ given by
\small
\begin{equation}\label{eq:orlicz}
	\|f\|^0_{\mathcal{M}_{\varphi,w}} = \sup\left\{ \int_I |fg| : \rho_{\varphi_*,w}(g) \leq 1\right\} = \sup\left\{ \int_I f^*g^* : \rho_{\varphi_*,w}(g) \leq 1\right\} = \inf_{k >0} \left(\frac{1}{k}(1 + P_{\varphi,w}(kf))\right).
\end{equation}
\normalsize
Now, we want to see when the infimum in (\ref{eq:orlicz}) is achieved. If $\varphi$ is an $N$-function, then $\varphi_*$ is also an $N$-function \cite{Kra} and its right-derivative $p(t) = \varphi'_+(t) \rightarrow \infty$ as $t \rightarrow \infty$. Also, we have $P_{\varphi,w}(f) = Q_{\varphi,w}(f)$ for every $f \in L_0$ \cite[Theorem 4.7]{KLR}. First, we define the interval $\bar{K}(f) = \left[{\bar{k}}^*, {\bar{k}}^{**}\right]$ and the constant $\bar{\theta}(f)$ for $f \in \mathcal{M}_{\varphi,w}$ by
\begin{align*}
	{\bar{k}}^*&= {\bar{k}}^*(f) = \inf\left\{k > 0 : \rho_{\varphi_*,w}\left(p\left(\frac{k(f^*)^0}{w}\right)\right) \geq 1 \right\}, \\
	{\bar{k}}^{**} &={\bar{k}}^{**}(f) = \sup\left\{k > 0 : \rho_{\varphi_*,w}\left(p\left(\frac{k(f^*)^0}{w}\right)\right) \leq 1\right\},\\
	\bar{\theta} &= \bar{\theta}(f) = \inf\left\{\lambda > 0 : P_{\varphi,w}\left(\frac{f}{\lambda}\right) < \infty\right\},
\end{align*}
In general, for $f \in \mathcal{M}_{\varphi,w}^0 \setminus \{0\}$ we see that $0 < \bar{k}^{*} \leq \bar{k}^{**} \leq \infty$. However, under our assumption on $\varphi$, the interval $\bar{K}(f)$ is always bounded, as we show in the next result.

	\begin{Lemma}\label{lem:NfuncinftyK}
		If $\varphi$ is an $N$-function, then  $\bar{k}^{**}(f) < \infty$ for every $f \in \mathcal{M}_{\varphi,w}$.
	\end{Lemma}
	
	\begin{proof}
		Assume to the contrary that $\bar{k}^{**} = \infty$. Then  $\rho_{\varphi_*,w}\left(p\left(\frac{k(f^*)^0}{w}\right)\right) \leq 1$ for every $k > 0$. Let $\tilde{f} = \sum_{i=1}^{n}c_n \chi_{[t_{i-1}, t_i)}$, where $t_0 = 0 < t_1 < t_2 < \cdots < t_n < \infty$ be a decreasing simple function with support of finite measure such that $\tilde{f} \leq f^*$ a.e. Denote $\tilde{F}(a, b) = \int_a^b \tilde{f}$. Then
		\[
		(\tilde{f})^0(t) = \sum_{j=1}^{m} \frac{\tilde{F}(t_{i_{j-1}, t_{i_j}})}{W(t_{i_{j-1}, t_{i_j}})}w\chi_{[t_{i_{j-1}}, t_{i_j})}(t),
		\]
		where $t_{i_0} = t_0 = 0$ and $t_{i_j} = t_s$ for some $s \in \{1, \cdots, n\}$, and $t_{i_1} < t_{i_2} < \cdots < t_{i_m}$. From Proposition \ref{prop:Halmono}, we have $p\left(\frac{k(\tilde{f})^0}{w}\right) \leq p\left(\frac{k(f^*)^0}{w}\right)$ a.e., and so 
		\begin{eqnarray*}
			\int_0^{t_{i_1}} \varphi_*\left(p\left(\frac{k(\tilde{f})^0}{w}\right)\right)w &=& \varphi_*\left(p\left(\frac{k\tilde{F}(t_{i_1})}{W(t_{i_1})}\right)\right)W(t_{i_1})\\
			&\leq&\sum_{j=1}^{m} \varphi_*\left(p\left(\frac{\tilde{F}(t_{i_{j-1}}, t_{i_{j}} )}{W(t_{i_{j-1}}, t_{i_{j}})}\right)\right)W(t_{i_{j-1}}, t_{i_{j}})\\
			&=& \int_I \varphi_*\left(p\left(\frac{k(\tilde{f})^0}{w}\right)\right)w\\
			&\leq& \int_I \varphi_*\left(p\left(\frac{k(f^*)^0}{w}\right)\right)w \leq 1, 
		\end{eqnarray*}
		where $(0, t_{i_1})$ is the first maximal level interval of $\tilde{f}$. Therefore,
		\begin{equation*}
			\varphi_*\left(p\left(\frac{k\tilde{F}(t_{i_1})}{W(t_{i_1})}\right)\right) \bigg{/} p\left(\frac{k\tilde{F}(t_{i_1})}{W(t_{i_1})}\right)\leq 1 \bigg{/} \left(p\left(\frac{k\tilde{F}(t_{i_1})}{W(t_{i_1})}\right) W(t_{i_1})\right).
		\end{equation*}
		Since $\varphi$ is an $N$-function, the complementary function $\varphi_*$ is also an $N$-function. So $\lim_{u \rightarrow \infty} \frac{\varphi_*(u)}{u} = \infty$. Moreover, $\lim_{t \rightarrow \infty} p(t) = \infty$. Hence the left-hand side goes to infinity as $k \rightarrow \infty$, but the right-hand side goes to $0$ as $k \rightarrow \infty$, which is a contradiction.
	\end{proof}

\begin{Lemma}\label{lem:relyoung}
	Let $f \in \mathcal{M}_{\varphi,w}$ and $k >0$, $n\in\mathbb{N}$. If $f_n = f\chi_{\{\frac{1}{n} < |f| \leq n\}}$, then $\rho_{\varphi_*,w}\left(p\left(\frac{k(f_n^*)^0}{w}\right)\right) < \infty$.    
\end{Lemma}

\begin{proof}
	Let $f \in \mathcal{M}_{\varphi,w}$ and define $f_n = f\chi_{\{\frac{1}{n} < |f| \leq n\}}$, $n \in \mathbb{N}$. Since $f$ is equimeasurable to $f^*$, by Proposition \ref{prop:distfin}, $m\{\frac{1}{n} < f^* \leq n\} \leq m\{ f^* > \frac{1}{n}\} < \infty$ for all $n \in \mathbb{N}$. Let $c = m\{\frac{1}{n} < f^* \leq n\}$. Hence each function $f_n$ is bounded with support of finite measure. Furthermore, notice that $f_n^* \leq n \chi_{(0,c)}$ a.e., so $(f_n^*)^0 \leq (n \chi_{(0,c)})^0$ a.e. by Proposition \ref{prop:Halmono}. Since $p$ is an increasing function, for $k > 0$ we have $p\left(\frac{k(f_n^*)^0}{w}\right) \leq p\left(\frac{k(n \chi_{(0,c)})^0}{w}\right)$ a.e. Hence
	\[
	\rho_{\varphi_*,w}\left(p\left(\frac{k(f_n^*)^0}{w}\right)\right) = \int_I \varphi_* \left(p\left(\frac{k(f_n^*)^0}{w}\right)\right)w \leq \int_I \varphi_* \left(p\left(\frac{k(n \chi_{(0,c)})^0}{w}\right)\right)w.
	\]
	
	Moreover,
	\[
	\int_I \varphi_* \left(p\left(\frac{k(n \chi_{(0,c)})^0}{w}\right)\right)w = \int_I \varphi_* \left(p\left(kn \frac{c}{W(c)}\chi_{(0,c)}\right)\right)w = \varphi_*(M)W(c), 
	\]
	
	where $M = p\left(kn \frac{c}{W(c)}\right)$. Since $W(c) < \infty$, we have $\rho_{\varphi_*,w}\left(p\left(\frac{k(f_n^*)^0}{w}\right)\right) < \infty$.  
\end{proof}

Before presenting a few basic properties of the Orlicz norm in $\mathcal{M}_{\varphi,w}$, we start with an auxillary lemma. The similar fact with respect to simple functions was shown in \cite{KLT2}, which was essential to prove the M-ideal properties in Orlicz-Lorentz spaces. Here we present a quick proof based on properties of level functions. 

\begin{Lemma}\label{lem:aux}
	Let $\varphi$ be an $N$-function and $f \in L_0$ such that $P_{\varphi,w}(f) < \infty$. Then the inverse level function $w^{f^*}$ satisfies $w^{f^*}\prec w$ and
	\begin{eqnarray*}
	P_{\varphi,w}(f) = Q_{\varphi,w}(f) &=& \int_I\varphi\left(\frac{(f^*)^0}{w}\right)w = \int_I\varphi\left(\frac{f^*}{w^{f^*}}\right)w^{f^*},\\	
	\rho_{\varphi_*,w}\left(p\left(\frac{(f^*)^0}{w}\right)\right) &=& \int_I \varphi_*\left(p\left(\frac{f^*}{w^{f^*}}\right)\right)w^{f^*} 
	\end{eqnarray*} 
\end{Lemma} 

\begin{proof}
	The first part and the fact that $w^{f^*} \prec w$ are  well-known \cite{KLR, KR2}. The second part comes from the definition of the inverse level function. Indeed, if we denote a countable family of maximal level intervals by $\{(a_i, b_i): i \in \mathbb{N}\}$, we see that
	\begin{eqnarray*}
		\int_I \varphi_*\left(p\left(\frac{f^*}{w^{f^*}}\right)\right)w^{f^*} &=& \int_{I \setminus \cup(a_i, b_i)}\varphi_*\left(p\left(\frac{f^*}{w}\right)\right)w + \sum_{i = 1}^{\infty}\varphi_*\left(p\left(\frac{F(a_i, b_i)}{W(a_i, b_i)}\right)\right)W(a_i,b_i)\\
		&=& \int_I \varphi_*\left(p\left(\frac{(f^*)^0}{w}\right)\right)w = \rho_{\varphi_*, w}\left(p\left(\frac{(f^*)^0}{w}\right)\right)
	\end{eqnarray*}
	
\end{proof}
Now, we provide some results on the Orlicz norm for $\mathcal{M}_{\varphi,w}$ that are similar to those for Orlicz-Lorentz spaces \cite{WC, WN}. 

\begin{Theorem}\label{th:OrliczM}
	Let $\varphi$ be an N-function. For $f\in \mathcal{M}_{\varphi,w}$,
	\begin{enumerate}[\rm(i)]
		\item $\|f\|_{\mathcal{M}_{\varphi,w}}\leq \|f\|_{\mathcal{M}_{\varphi,w}}^0 \leq 2\|f\|_{\mathcal{M}_{\varphi,w}}$.
		\item If there exists $k>0$ such that $\rho_{\varphi_*,w} \left(p\left(\frac{k(f^*)^0}{w}\right)\right) = 1$, then $\|f\|_{\mathcal{M}_{\varphi,w}}^0 = \int_I f^*p\left(\frac{k(f^*)^0}{w}\right) = \frac{1}{k}(1 + P_{\varphi,w} (kf))$.
		\item $k \in \bar{K}(f)$ if and only if $\|f\|_{\mathcal{M}_{\varphi,w}}^0 = \frac{1}{k}(1 + P_{\varphi,w}(kf))$. 
	\end{enumerate}
\end{Theorem}

\begin{proof}
	
	(i): For $g \in \Lambda_{\varphi_*,w}$, we have $\|g\|_{\Lambda_{\varphi_*,w}} \leq \|g\|_{\Lambda_{\varphi_*,w}}^0 \leq 2\|g\|_{\Lambda_{\varphi_*,w}}$ \cite{WC}. Hence
	\[
	\|f\|_{\mathcal{M}_{\varphi,w}} = \sup\left\{\int_I |fg| : \|g\|_{\Lambda_{\varphi_*,w}}^0 \leq 1\right\} \leq \sup\left\{\int_I |fg| : \|g\|_{\Lambda_{\varphi_*,w}} \leq 1\right\} = \|f\|_{\mathcal{M}_{\varphi,w}}^0.
	\]
	
	By Proposition \ref{lem:basiclux}, we see that $P_{\varphi,w}\left(\frac{f}{\|f\|_{\mathcal{M}_{\varphi,w}}}\right) \leq 1$. Then we have
	\[
	\left\|\frac{f}{\|f\|_{\mathcal{M}_{\varphi,w}}}\right\|_{\mathcal{M}_{\varphi,w}}^0 \leq 1 + P_{\varphi,w}\left(\frac{f}{\|f\|_{\mathcal{M}_{\varphi,w}}}\right) \leq 2.
	\]
	Therefore, $\|f\|_{\mathcal{M}_{\varphi,w}}\leq \|f\|_{\mathcal{M}_{\varphi,w}}^0 \leq 2\|f\|_{\mathcal{M}_{\varphi,w}}$.
	
	(ii): We will use the inverse level functions to show our claim. Denote the inverse level function of $f^*$ by $w^{f^*}$. Assume $k> 0$ such that $\rho_{\varphi_*,w} \left(p\left(\frac{k(f^*)^0}{w}\right)\right) = 1$. By Lemma \ref{lem:aux} we see that
	\begin{equation}\label{eq: modeq1}
		\int_I \varphi_*\left(p\left(\frac{kf^*}{w^{f^*}}\right)\right)w^{f^*} = \rho_{\varphi_*,w}\left(p\left(\frac{k(f^*)^0}{w}\right)\right) = 1.
	\end{equation}
	
	Observe that 
	\[
	\int_I f^* p\left(\frac{k(f^*)^0}{w}\right) = \int_I f^* p\left(\frac{kf^*}{w^{f^*}}\right) = \frac{1}{k} \int_I \frac{kf^*}{w^{f^*}} p\left(\frac{kf^*}{w^{f^*}}\right)w^{f^*}.
	\]  
	From Young's equality \cite[pp. 8, formula (3)]{Chen},
	\begin{equation}\label{rhop}
		\frac{1}{k} \int_I \frac{kf^*}{w^{f^*}} p\left(\frac{kf^*}{w^{f^*}}\right)w^{f^*} = \frac{1}{k} \left( \int_I \varphi\left(\frac{kf^*}{w^{f^*}}\right)w^{f^*} + \int_I \varphi_*\left(p\left(\frac{kf^*}{w^{f^*}}\right)\right)w^{f^*}\right), 
	\end{equation}
	and by Lemma \ref{lem:aux} and (\ref{eq: modeq1}), we obtain 
	\begin{equation}\label{eq:equal}
		\int_I f^* p\left(\frac{k(f^*)^0}{w}\right) = \frac{1}{k}( 1 + P_{\varphi,w}(kf)).
	\end{equation} 
	From the definition of the Orlicz norm for $\mathcal{M}_{\varphi,w}$ and from the fact that $\rho_{\varphi_*,w} \left(p\left(\frac{k(f^*)^0}{w}\right)\right) = 1$, we have 
	\[
	\|f\|_{\mathcal{M}_{\varphi,w}}^0 \geq \int_I f^* p\left(\frac{k(f^*)^0}{w}\right).
	\]
	
	To show the reverse inequality, since $\mathcal{M}_{\varphi,w}$ is r.i., we have
	\[
	\|f\|_{\mathcal{M}_{\varphi,w}}^0 = \sup\left\{\int_I f^*g^* : \rho_{\varphi_*, w}(g) \leq 1\right\} =  \frac{1}{k} \sup\left\{\int_I k f^*g^* : \rho_{\varphi_*, w}(g) \leq 1\right\}.
	\]
	In addition, by Young's inequality and by Lemma \ref{lem:aux}, we see that
	\[
	\int_I k f^*g^* = \int_I \frac{kf^*g^*}{w^{f^*}}w^{f^*} \leq \int_I \varphi\left(\frac{kf^*}{w^{f^*}}\right)w^{f^*} + \int_I \varphi_*(g^*)w^{f^*} \leq P_{\varphi,w}(kf) + \rho_{\varphi_*,w}(g).
	\]
	Hence for $g \in \Lambda_{\varphi_*,w}$ such that $\rho_{\varphi_*,w}(g) \leq 1$, by (\ref{eq:equal}), 
	\[
	\|f\|_{\mathcal{M}_{\varphi,w}}^0 = \frac{1}{k} \sup\left\{\int_I k f^*g^* : \rho_{\varphi_*, w}(g) \leq 1\right\} \leq \frac{1}{k}(P_{\varphi,w}(kf) +1)) = \int_I f^* p\left(\frac{k(f^*)^0}{w}\right),
	\]
	and this proves our claim.
	
	(iii): For $f \in \mathcal{M}_{\varphi,w}^0$, define a function $T(k) = \frac{1}{k}(1 + P_{\varphi,w}(kf))$. Let $\bar{\theta} = \bar{\theta}(f)$. The function $T(k)$ is continuous on the interval $(0, 1/\bar{\theta})$. We first want to show that $\bar{k}^{**} < 1/\bar{\theta}$. Notice that for every $k < \bar{k}^{**}$, $\rho_{\varphi_*,w}\left(p\left(\frac{k(f^*)^0}{w}\right)\right) \leq 1$. In view of (\ref{rhop}), 
	 
	\[
	\int_I kf^* p\left(\frac{k(f^*)^0}{w}\right) =  P_{\varphi,w}(kf) +  \rho_{\varphi_*,w}\left(p\left(\frac{k(f^*)^0}{w}\right)\right).
	\]
	Hence
	\[
	T(k) = \frac{1}{k}(1 + P_{\varphi,w}(kf)) \leq \frac{1}{k}\left(1 + \int_I kf^* p\left(\frac{k(f^*)^0}{w}\right)\right) \leq \frac{1}{k}(1 + \|kf\|^0_{\mathcal{M}_{\varphi,w}}),   
	\]
	and this shows that $T(k) \leq \frac{1}{k} + \|f\|^0_{\mathcal{M}_{\varphi,w}}$ for every $k < \bar{k}^{**}$. Let $(k_n)_{n=1}^{\infty}$ be a sequence of real numbers such that $k_n \uparrow \bar{k}^{**}$. Then by Fatou's lemma, we obtain
	\[
	T(\bar{k}^{**}) \leq \liminf_n\frac{1}{k_n}(1+ P_{\varphi,w}(k_nf)) \leq  \lim_{n\rightarrow \infty} \frac{1}{k_n} + \|f\|^0_{\mathcal{M}_{\varphi,w}} = 1/\bar{k}^{**} + \|f\|^0_{\mathcal{M}_{\varphi,w}}.
	\]
	Since $\varphi$ is an $N$-function, $\bar{k}^{**} < \infty$ by Lemma \ref{lem:NfuncinftyK}. Hence by the statement (i),
	\[
	P_{\varphi,w}(\bar{k}^{**}f) \leq \|\bar{k}^{**}f\|_{\mathcal{M}_{\varphi,w}} \leq \|\bar{k}^{**}f\|^0_{\mathcal{M}_{\varphi,w}} < \infty.
	\]
	Thus, we get $\bar{k}^{**} < 1/\bar{\theta}$.
	
	Now, let $k_1, k_2 \in (0, 1/\bar{\theta})$ such that $k_1 > k_2$. Define $f_n = f\chi_{\{\frac{1}{n} < |f| \leq n\}}$. Then by Young's inequality, for every $n\in\mathbb{N}$,
	\small
	\begin{flalign*}
		\int_I k_1f_n^* p\left(\frac{k_2 (f_n^*)^0}{w}\right) = \int_I \frac{k_1f_n^*}{w^{f_n^*}}p\left(\frac{k_2 f_n^*}{w^{f_n^*}}\right)w^{f_n^*} &\leq \int_I \varphi\left(\frac{k_1f_n^*}{w^{f_n^*}}\right)w^{f_n^*} + \int_I \varphi_*\left(p\left(\frac{k_2 f_n^*}{w^{f_n^*}}\right)\right)w^{f_n^*}\\
		&= P_{\varphi,w}(k_1f_n) + \rho_{\varphi_*,w}\left(p\left(\frac{k_2 (f_n^*)^0}{w}\right)\right).
	\end{flalign*}
	
	\normalsize
	
	Since $\rho_{\varphi_*,w}\left(p\left(\frac{k_2 (f_n^*)^0}{w}\right)\right) < \infty$ by Lemma \ref{lem:relyoung}, 
	\begin{equation}\label{eq:Young1}
		P_{\varphi,w}(k_1f_n) \geq \int_I k_1f_n^* p\left(\frac{k_2 (f_n^*)^0}{w}\right) -  \rho_{\varphi_*,w}\left(p\left(\frac{k_2 (f_n^*)^0}{w}\right)\right).
	\end{equation}
	In addition by Young's equality and Lemma \ref{lem:aux}, we see that
	\small
	\begin{flalign*}
		\int_I k_2f_n^* p\left(\frac{k_2 (f_n^*)^0}{w}\right) = \int_I \frac{k_2f_n^*}{w^{f_n^*}}p\left(\frac{k_2 f_n^*}{w^{f_n^*}}\right)w^{f_n^*} &= \int_I \varphi\left(\frac{k_2f_n^*}{w^{f_n^*}}\right)w^{f_n^*} + \int_I \varphi_*\left(p\left(\frac{k_2 f_n^*}{w^{f_n^*}}\right)\right)w^{f_n^*}\\
		&= P_{\varphi,w}(k_2f_n) + \rho_{\varphi_*,w}\left(p\left(\frac{k_2 (f_n^*)^0}{w}\right)\right).
	\end{flalign*}
	\normalsize
	Then we obtain 
	\begin{equation}\label{eq:Young2}
		P_{\varphi,w}(k_2f_n) = \int_I k_2f_n^* p\left(\frac{k_2 (f_n^*)^0}{w}\right) -  \rho_{\varphi_*,w}\left(p\left(\frac{k_2 (f_n^*)^0}{w}\right)\right).
	\end{equation}
	Observe that
	\[
	\frac{1}{k_1}(1 + P_{\varphi,w}(k_1 f_n)) - \frac{1}{k_2}(1 + P_{\varphi,w}(k_2 f_n)) = \frac{k_2 - k_1}{k_1 k_2} + \frac{1}{k_1}P_{\varphi,w}(k_1 f_n) - \frac{1}{k_2}P_{\varphi,w}(k_2 f_n) .
	\]
	By adding and subtracting by $\frac{1}{k_1}P_{\varphi,w}(k_2 f_n)$,  
	\begin{multline*}
		\frac{k_2 - k_1}{k_1 k_2} + \frac{1}{k_1}P_{\varphi,w}(k_1 f_n) - \frac{1}{k_2}P_{\varphi,w}(k_2 f_n) \\ = \frac{k_1 - k_2}{k_1 k_2}\left( -1 + \frac{k_2}{k_1 - k_2}\left(P_{\varphi,w}(k_1 f_n) - P_{\varphi,w}(k_2 f_n)\right)- P_{\varphi,w}(k_2 f_n)\right).
	\end{multline*}
	From (\ref{eq:Young1}) and (\ref{eq:Young2}), 
	\begin{multline*}
		\frac{k_1 - k_2}{k_1 k_2}\left( -1 + \frac{k_2}{k_1 - k_2}(P_{\varphi,w}(k_1 f_n) - P_{\varphi,w}(k_2 f_n))- P_{\varphi,w}(k_2 f_n)\right) \\ \geq \frac{k_1 - k_2}{k_1 k_2}\left(-1 + \frac{k_2}{k_1 - k_2}\left(\int_I (k_1 - k_2) f_n p\left(\frac{k_2 (f_n^*)^0}{w}\right)\right) - P_{\varphi,w}(k_2 f_n)\right) \\ = \frac{k_1 - k_2}{k_1 k_2}\left(\rho_{\varphi_*,w}\left(p\left(\frac{k_2 (f_n^*)^0}{w}\right)\right) - 1\right).
	\end{multline*}
	Hence, we have
	\[
	\frac{1}{k_1}(1 + P_{\varphi,w}(k_1 f_n)) - \frac{1}{k_2}(1 + P_{\varphi,w}(k_2 f_n))  \geq \frac{k_1 - k_2}{k_1 k_2}\left(\rho_{\varphi_*,w}\left(p\left(\frac{k_2 (f_n^*)^0}{w}\right)\right) - 1\right).
	\]
	
 	From the fact that $(f_n^*)^0 \uparrow (f^*)^0$ a.e. by Proposition \ref{prop:Halmono} and by the Monotone Convergence Theorem, notice that for $k_1 > k_2$,
	\begin{equation}\label{eq:increasing}
		T(k_1) - T(k_2) \geq \frac{k_1 - k_2}{k_1 k_2}\left(\rho_{\varphi_*,w}\left(p\left(\frac{k_2 (f^*)^0}{w}\right)\right) - 1\right).
	\end{equation}
	By the same argument, if $k_1 < k_2$, we can also show that 
	\begin{equation}\label{eq:decreasing}
		T(k_1) - T(k_2) \geq \frac{k_1 - k_2}{k_1 k_2}\left(\rho_{\varphi_*,w}\left(p\left(\frac{k_1 (f^*)^0}{w}\right)\right) - 1\right).
	\end{equation}
	
	For $0 < k_1 < k_2 < \bar{k}^{*}$, $\frac{k_1 - k_2}{k_1k_2} < 0$ and $\rho_{\varphi_*,w}\left(p\left(\frac{k_1 (f^*)^0}{w}\right)\right) < 1$. So from (\ref{eq:decreasing}), $T(k)$ is decreasing on $(0, \bar{k}^{*})$. Moreover, since $\frac{k_1 - k_2}{k_1k_2} >0$ and $\rho_{\varphi_*,w}\left(p\left(\frac{k_2 (f^*)^0}{w}\right)\right) > 1$ for $\bar{k}^{**} < k_1 < k_2 < 1/\bar{\theta}$, the function $T(k)$ is increasing on $(\bar{k}^{**}, 1/\bar{\theta})$ by (\ref{eq:increasing}). Notice from the definitions of $\bar{k}^{*}$ and $\bar{k}^{**}$ that $\rho_{\varphi_*,w}\left(p\left(\frac{k (f^*)^0}{w}\right)\right) = 1$ for any $k \in (\bar{k}^{*},\bar{k}^{**})$. Hence, 
	\[
	\|f\|^0_{\mathcal{M}_{\varphi,w}} = \inf_{l >0} T(l)= \frac{1}{k}(1 + P_{\varphi,w}(kf))
	\]
	by Theorem \ref{th:OrliczM}.(ii). Observe that $T(l) > T(\bar{k}^*)$ for every $l < \bar{k}^*$ and $T(l) > T(\bar{k}^{**})$ for every $l > \bar{k}^{**}$. Since $T(l)$ is continuous on the interval $(0, 1/\bar{\theta})$ and $\bar{k}^{*}, \bar{k}^{**} \in (0, 1/\bar{\theta})$, $\|f\|^0_{\mathcal{M}_{\varphi,w}} = \inf_{l >0} T(l) = T(\bar{k}^*) = T(\bar{k}^{**})$. Hence, if $k \in \bar{K}(f) = [\bar{k}^*, \bar{k}^{**}]$, then $\|f^0\|_{\mathcal{M}_{\varphi,w}} = \frac{1}{k}(1 + P_{\varphi,w}(kf))$.
	
	To show the converse, define $T(k)$ and $\bar{\theta}$ as before. Recall the fact that $T(k)$ is continuous on the interval $(0, 1/\bar{\theta})$. Let $k_0 > 0$ be such that $\|f\|^0_{\mathcal{M}_{\varphi,w}}  =  T(k_0) = \inf_{k > 0}T(k)$. If $k_0 \in (0, \bar{k}^*)$, $T(k_0) > T(\bar{k}^*)$ because $T(k)$ is decreasing on the interval $(0, \bar{k}^*)$ by (\ref{eq:decreasing}). Also, $T(k_0) >  T(\bar{k}^{**})$ for $k_0 \in (\bar{k}^{**}, 1/\bar{\theta})$  since $T(k)$ is increasing on the interval $(\bar{k}^{**}, 1/\bar{\theta})$ by (\ref{eq:increasing}). Hence $T(k_0) = \|f\|^0_{\mathcal{M}_{\varphi,w}}$ only when $k_0 \in \bar{K}(f)$.
\end{proof}

\section{Applications to Orlicz-Lorentz spaces}
The characterization of separable $\mathcal{M}_{\varphi,w}$ spaces in Section 4 allows us to explore various properties in Orlicz-Lorentz spaces. First, we can characterize the reflexivity of Orlicz-Lorentz spaces and their K\"othe duals.
	
	\begin{Theorem}\label{th:reflexiveOL}
		The following statements are equivalent:
		\begin{enumerate}[\rm(i)]
			\item An Orlicz-Lorentz space $\Lambda_{\varphi,w}$ is reflexive.
			\item $\varphi$ and its complementary function $\varphi_*$ satisfy the appropriate $\Delta_2$-condition.  
		\end{enumerate}
		The analogous statement for $\lambda_{\varphi,w}$ also holds.
	\end{Theorem}  
	
	\begin{proof}
		An Orlicz function $\varphi$ satisfies the appropriate $\Delta_2$-condition if and only if $\Lambda_{\varphi, w}$ is order-continuous by \cite[Theorem 2.4]{K}.  Also the complementary function $\varphi_*$ satisfies the appropriate $\Delta_2$-condition if and only if $(\Lambda_{\varphi, w})' \simeq\mathcal{M}_{\varphi_*,w}^0$ is order-continuous by Theorem \ref{th:copy}, Theorem \ref{th:KLR}, and the equivalence between the Orlicz norm and the Luxemburg norm. From the fact that a Banach function lattice $X$ is reflexive if and only if both $X$ and $X'$ are order-continuous \cite[Corollary 1.4.4]{BS}, we obtain the desired result. 
	\end{proof}
		
	\begin{Theorem}
		The following statements are equivalent:
		\begin{enumerate}[\rm(i)]
			\item The space $\mathcal{M}_{\varphi, w}$ is reflexive.
			\item $\varphi$ and its complementary function $\varphi_*$ satisfy the appropriate $\Delta_2$-condition.
		\end{enumerate}
		The analogous statement for $\mathfrak{m}_{\varphi,w}$ also holds.
	\end{Theorem} 
	\begin{proof}
	An Orlicz function $\varphi$ satisfies the appropriate $\Delta_2$-condition if and only if $(\mathcal{M}_{\varphi_*,w})' \simeq \Lambda_{\varphi, w}^0$ is order-continuous by \cite[Theorem 2.4]{K}, Theorem \ref{th:KLR}, and the equivalence between the Orlicz norm and the Luxemburg norm. Also the complementary function $\varphi_*$ satisfies the appropriate $\Delta_2$-condition if and only if $\mathcal{M}_{\varphi_*,w}$ is order-continuous by Theorem \ref{th:copy}. From the fact that a Banach function lattice $X$ is reflexive if and only if both $X$ and $X'$ are order-continuous \cite[Corollary 1.4.4]{BS}, we obtain the desired result.
	\end{proof}
	
\subsection{M-embedded Orlicz-Lorentz spaces} 

A closed subspace $Y$ of $X$ is said to be an {\it M-ideal} if there exists a projection $P:X^* \rightarrow X^*$ such that the $P(X^*) = Y^{\perp}$ and $\|x^*\| = \|Px^*\| + \|(I-P)x^*\|$. A Banach space $X$ is said to be {\it $M$-embedded} if $X$ is an $M$-ideal in its bidual $X^{**}$. Assuming that $X$ is an $M$-ideal in its bidual $X^{**}$, if $Y$ is a separable closed subspace of $X$, then $Y^*$ is separable \cite[Theorem 2.6]{Lima}.

The dual $X^*$ of a Banach function lattice $X$ is isometrically isomorphic to the K\"othe dual space $X'$ if and only if $X$ is order-continuous \cite[Corollary 1.4.3]{BS}. For the case of Orlicz spaces equipped with the Luxemburg norm, it is shown that the order-continuous subspace $(L_\varphi)_a$ is $M$-embedded if $\varphi_*$ satisfies the appropriate $\Delta_2$-condition \cite{HWW}. As a matter of fact, the analogous statement is also true for $(\Lambda_{\varphi,w})_a$ and $(\lambda_{\varphi,w})_a$. 
	
Now, we are ready to provide sufficient conditions for M-embeddedness of $(\Lambda_{\varphi, w})_a$.   

\begin{Theorem}\label{MembedF}
\begin{enumerate}[\rm(i)]
	\item If both $\varphi$ and $\varphi_*$ satisfy the appropriate $\Delta_2$-condition, then the order-continuous subspace $(\Lambda_{\varphi,w})_a$ is an $M$-ideal in its bidual $((\Lambda_{\varphi,w})_a)^{**} \simeq \Lambda_{\varphi,w}$.
	\item If neither $\varphi$ nor $\varphi_*$ satisfies the appropriate $\Delta_2$-condition, then the order-continuous subspace $(\Lambda_{\varphi,w})_a$ is not an $M$-ideal in its bidual.
	\item If $\varphi$ does not satisfy the appropriate $\Delta_2$-condition while $\varphi_*$ does, then the order-continuous subspace $(\Lambda_{\varphi,w})_a$ is an $M$-ideal in its bidual $((\Lambda_{\varphi,w})_a)^{**} \simeq \Lambda_{\varphi,w}$.
\end{enumerate}
\end{Theorem}

\begin{proof}
	(i): In view of Theorem \ref{th:reflexiveOL}, $(\Lambda_{\varphi,w})_a = \Lambda_{\varphi,w} \simeq (\Lambda_{\varphi,w})_a^{**}$. Hence, the statement holds trivially.   
	
	(ii): Suppose that both $\varphi$ and $\varphi_*$ does not satisfy the appropriate $\Delta_2$-condition and assume to the contrary that $(\Lambda_{\varphi,w})_a$ is an $M$-ideal in its bidual $((\Lambda_{\varphi,w})_a)^{**}$.  In view of \cite[Theorem 2.6]{Lima}, the dual of $(\Lambda_{\varphi,w})_a$ has to be separable. From the fact that $X^*$ is isometrically isomorphic to $X'$ for order-continuous Banach function spaces, we have $((\Lambda_{\varphi,w})_a)^* \simeq (\Lambda_{\varphi,w})' \simeq \mathcal{M}_{\varphi_*,w}^0$ by Theorem \ref{th:KLR}. However, the space $\mathcal{M}_{\varphi_*,w}^0$ is not separable by Corollary \ref{cor:sep}, which is a contradiction.
	
	(iii): It was shown in \cite[Theorem 3.10]{KLT3} that the order-continuous subspace $(\Lambda_{\varphi,w})_a$ is an $M$-ideal in $\Lambda_{\varphi,w}$. Since $ ((\Lambda_{\varphi,w})_a)^{**} \simeq (\mathcal{M}_{\varphi_*,w}^0)^* \simeq \Lambda_{\varphi,w}$ by \cite[Corollary 1.4.3]{BS} and Theorem \ref{th:KLR}, the claim holds.  
\end{proof}

The sequence analogue of Theorem \ref{MembedF} can be obtained by a similar argument. Hence we only provide the statement here. For the proof of \ref{MembedS}.(iii), we refer to \cite[Theorem 3.11]{KLT3}. 

\begin{Theorem}\label{MembedS}
	\begin{enumerate}[\rm(i)]
		\item If both $\varphi$ and $\varphi_*$ satisfy the $\Delta_2^0$-condition, then the order-continuous subspace $(\lambda_{\varphi,w})_a$ is an $M$-ideal in its bidual $((\lambda_{\varphi,w})_a)^{**} \simeq \lambda_{\varphi,w}$.
		\item If both $\varphi$ and $\varphi_*$ do not satisfy the $\Delta_2^0$-condition, then the order-continuous subspace $(\lambda_{\varphi,w})_a$ is not an $M$-ideal in its bidual $((\lambda_{\varphi,w})_a)^{**}$.
		\item If $\varphi$ does not satisfy the $\Delta_2^0$-condition while $\varphi_*$ does, then the order-continuous subspace $(\lambda_{\varphi,w})_a$ is an $M$-ideal in its bidual $((\lambda_{\varphi,w})_a)^{**} \simeq \lambda_{\varphi,w}$.
	\end{enumerate}
\end{Theorem}

\subsection{Uniqueness of norm-preserving extension on Orlicz-Lorentz spaces equipped with the Orlicz norm}

By the Hahn-Banach extension theorem, any bounded linear functional on a subspace of a Banach space has a norm-preserving extension to the whole space. In this section, we study the uniqueness of such extension from $(\Lambda_{\varphi,w}^0)_a$ to $\Lambda_{\varphi,w}^0$. In view of \cite[Proposition 1.12]{HWW}, we can deduce that every integral functional $H \in (\Lambda_{\varphi,w})_a^*$ has a unique norm-preserving extension to $\Lambda_{\varphi,w}$ because $(\Lambda_{\varphi,w})_a$ is an $M$-ideal in $\Lambda_{\varphi,w}$. Hence, we only consider when the space is equipped with the Orlicz norm. 

Let us recall the interval $K(f) = [k^*, k^{**}]$ for $f \in \Lambda_{\varphi,w}^0$ \cite{WC}, where   
	\begin{align*}
		k^* &= k^*(f) = \inf\{k > 0 : \rho_{\varphi_*,w}(p(kf)) \geq 1 \},\, \text{and}\\ 
		k^{**} &= k^{**}(f) = \sup\{ k >0 : \rho_{\varphi_*,w}(p(kf)) \leq 1\}.
	\end{align*}  
In general, it is well-known that $0 \leq k^* \leq k^{**} \leq \infty$. By the similar reasoning as Lemma \ref{lem:NfuncinftyK}, we can show that $k^{**} < \infty$ when $\varphi$ is an Orlicz $N$-function. We can also define $k^*(x)$, $k^{**}(x)$, and $K(x)$ for $x \in \lambda_{\varphi,w}$ by replacing $\rho_{\varphi,w}$ with $\alpha_{\varphi,w}$.

Notice that $k^*$ and $k^{**}$ for $\Lambda_{\varphi, w}^0$ as well as $\bar{k}^{*}$ and $\bar{k}^{**}$ for $\mathcal{M}_{\varphi,w}^0$ in section 6 are defined by the same modular $\rho_{\varphi_*, w}$. Such confusion comes from the fact that investigation on the space $\mathcal{M}_{\varphi,w}$ was initiated much later than \cite{WC,WN}. For Orlicz spaces $L_{\varphi}^0$, this is not an issue because the K\"othe dual space of $L_{\varphi}^0$ is $L_{\varphi_*}$, which is another Orlicz space. However, as we mentioned earlier, the K\"othe dual space of an Orlicz-Lorentz space is not precisely an Orlicz-Lorentz space, and its modular shows different behaviors from the one for Orlicz-Lorentz spaces. Now, in view of Proposition \ref{prop:declevel}, we see that
\[
\rho_{\varphi_*, w}(p(kf)) = \rho_{\varphi_*, w}(p(kf^*)) = P_{\varphi_*, w}(p(kf^*)w).
\]
Hence, we provide equivalent definitions of $k^*$ and $k^{**}$ in terms of the modular $P_{\varphi,w}$ to be consistent with the K\"othe duality between Orlicz-Lorentz spaces and the spaces $\mathcal{M}_{\varphi,w}$ as follows.
\begin{align*}
	k^* &= k^*(f) = \inf\{k > 0 : P_{\varphi_*,w}(p(kf^*)w) \geq 1 \},\, \text{and}\\ 
	k^{**} &= k^{**}(f) = \sup\{ k >0 : P_{\varphi_*,w}(p(kf^*)w) \leq 1\}.
\end{align*}  
With the modular $\mathfrak{p}_{\varphi_*,w}$, we can similarly define the constants $k^*$ and $k^{**}$ for Orlicz-Lorentz sequence space $\lambda_{\varphi,w}^0$. 
 
The following results show when the infimum for the Orlicz norm $\|\cdot\|_{\varphi,w}^0$ is attained.

\begin{Theorem}\cite[pg 133]{WC}\label{WC} {\rm{(}c.f. \cite[Proposition 1.5 and 1.8]{WN} for the sequence case\rm{)}}
Let $\varphi$ be an $N$-fucntion.
	\begin{enumerate}[\rm(i)]
		\item If there exists $k>0$ such that $\rho_{\varphi_*,w} (p(k|f|)) = 1$, then $\|f\|_{\varphi,w}^0 = \int_0^{\gamma} f^*p(kf^*)w = \frac{1}{k}(1 + \rho_{\varphi,w} (kf))$.
		\item $k \in K(f)$ if and only if $\|f\|_{\varphi,w}^0 = \frac{1}{k}(1 + \rho_{\varphi,w}(kf))$.
	\end{enumerate}
	The analogous statements occur in Orlicz-Lorentz sequence space when the modular $\rho_{\varphi,w}$ is replaced by the modular $\alpha_{\varphi,w}$.
\end{Theorem}

In order to study the existence of a unique norm-preserving extension of a bounded linear functional on $(\Lambda_{\varphi, w}^0)_a$ to $\Lambda_{\varphi, w}^0$, we need to know when a bounded linear functional is norm-attaining. A bounded linear functional $F \in X^*$ on $X$ is said to be {\it norm-attaining} if $|F(f)| = \|F\|_{X^*}$ for some $f \in X$ such that $\|f\|_X = 1$. For Orlicz-Lorentz spaces, it is useful to know an explicit formula for the norm of a bounded linear functional to characterize its norm-attainment. 

\begin{Theorem}\cite[Theorem 3.6]{KLT2} {\rm{(}c.f. \cite[Theorem 3.7]{KLT2} for the sequence case\rm{)}}\label{Orlicz}
	Let $\varphi$ be an $N$-function and $F$ be a bounded linear functional on $\Lambda^0_{\varphi,w}$. Then $F= H + S$, where $H(f) = \int_I fh$ for some
	$h\in \mathcal{M}_{\varphi_*,w}$,
	$\|H\|= \|h\|_{\mathcal{M}_{\varphi_*,w}}$, $S(f)=0$
	for all $f\in (\Lambda_{\varphi,w})_a$, and   $\|F\| = \inf\{\lambda>0 : P_{\varphi_*,w}(\frac{h}{\lambda}) + \frac{1}{\lambda}\|S\| \leq 1\}$. The similar formula also holds for the sequence space $\lambda_{\varphi,w}$.
\end{Theorem}

\begin{Theorem}\cite[Theorem 2.2]{KLT2}\label{theta}
	For any singular functional $S$ of $\Lambda_{\varphi,w}$ equipped with the Luxemburg norm or the Orlicz norm, $\|S\| = \|S\|_{(\Lambda_{\varphi,w})^*} = \|S\|_{(\Lambda_{\varphi,w}^0)^*} = \sup\{S(f) : \rho_{\varphi,w}(f) < \infty\} = \sup\{\frac{S(f)}{\theta(f)} : f \in \Lambda_{\varphi,w} \setminus (\Lambda_{\varphi,w})_a\}$.
	The analogous formulas hold for Orlicz-Lorentz sequence spaces.
\end{Theorem}

Now, we are ready to provide a characterization for the norm-attaining functionals on $\Lambda_{\varphi,w}^0$ and $\lambda_{\varphi,w}^0$. 

\begin{Theorem}\label{th:suppfunc}
	Let $\varphi$ be an $N$-function and let $F = H + S$ be a bounded linear functional on $\Lambda_{\varphi,w}^0$, where $H(f) = \int_I fh$ for some $h \in \mathcal{M}_{\varphi_*,w}$ and $S(f) = 0$ at $f \in (\Lambda_{\varphi,w}^0)_a$. Then $F$ is norm-attaining if and only if there exists $f \in \Lambda_{\varphi, w}^0$ with $\|f\|_{\varphi,w}^0 = 1$ such that for some $k \in K(f)$, the following conditions are satisfied 
	\begin{enumerate}[\rm(i)]
		\item $P_{\varphi_*,w}(\frac{h}{\|F\|}) + \frac{\|S\|}{\|F\|} = 1$,
		\item $\|S\| = S(kf)$,
		\item $\int_I hf= \int_I kh^*f^* = \int_I (h^*)^0f^*$ and $\int_I \frac{k(h^*)^0f^*}{\|F\|}= \rho_{\varphi,w}(kf) + P_{\varphi_*,w}(\frac{h}{\|F\|}).$
	\end{enumerate} 
\end{Theorem}

\begin{proof}
	Let $f \in \Lambda_{\varphi,w}^0$ be such that $\|f\|_{\varphi,w}^0 = 1$, $\|F\|= F(f)$ and $k \in K(f)$. Then we have $\rho_{\varphi,w}(kf) < \infty$. Indeed, notice that $1 = \|f\|^0 = \frac{1}{k} (1 + \rho_{\varphi,w}(kf))$ for $k \in K(f)$ from Theorem \ref{WC}. Hence $\rho_{\varphi,w}(kf) = k-1 < k^{**}(f) < \infty$. Let $h \in \mathcal{M}_{\varphi_*,w}$. From the fact that $h^* \prec (h^*)^0$, the Young's inequality and Lemma \ref{lem:hardy} give us  
	\begin{eqnarray*}
		1 = \frac{F(f)}{\|F\|} = \frac{H(f)}{\|F\|} + \frac{S(f)}{\|F\|} &\leq& \frac{\int_I h^*f^*}{\|F\|} + \frac{S(f)}{\|F\|} \leq \frac{\int_I (h^*)^0f^*}{\|F\|} + \frac{S(f)}{\|F\|}\\
		&=& \frac{1}{k}\left(\int_I \frac{k(h^*)^0f^*}{w\|F\|}w + \frac{S(kf)}{\|F\|}\right)\\
		&\leq& \frac{1}{k}\left(\int_I \varphi(kf^*)w + \int_I \varphi_*\left(\frac{(h^*)^0}{w\|F\|}\right)w + \frac{S(kf)}{\|F\|}\right)\\
		&=&\frac{1}{k}\left(\rho_{\varphi,w}(kf) + P_{\varphi_*,w}\left(\frac{h}{\|F\|}\right) + \frac{S(kf)}{\|F\|}\right).
	\end{eqnarray*} 
	Now, we claim that
	\begin{equation}\label{welld}
		P_{\varphi_*,w}\left(\frac{h}{\|F\|}\right) + \frac{\|S\|}{\|F\|} \leq 1.
	\end{equation} 
	Indeed, in view of Theorem \ref{Orlicz}, let $(\lambda_n)$ be a sequence of real numbers such that $(\lambda_n) \downarrow \|F\|$ and $P_{\varphi_*,w}(\frac{h}{\lambda_n}) + \frac{\|S\|}{\lambda_n} \leq 1$ for every $n \in \mathbb{N}$. Let $g(k) = P_{\varphi_*,w}(kh) + k\|S\|$ for $k>0$. The function $g(k)$ is increasing and continuous on the interval $(0, 1/\bar{\theta})$, where $\bar{\theta} = \bar{\theta}(h) = \inf\left\{\lambda >0 : P_{\varphi_*,w}\left(\frac{h}{\lambda}\right) < \infty\right\}$. Notice that $P_{\varphi_*,w}(\frac{h}{\lambda_n}) + \frac{\|S\|}{\lambda_n} \leq 1$ for every $n \in \mathbb{N}$. Hence 
	\[
	\lim_{n \rightarrow \infty}P_{\varphi_*,w}\left(\frac{h}{\lambda_n}\right) + \frac{\|S\|}{\|F\|} = \lim_{n \rightarrow \infty} \left(P_{\varphi_*,w}\left(\frac{h}{\lambda_n}\right) + \frac{\|S\|}{\lambda_n}\right) \leq 1.
	\]
	Since $\frac{1}{\lambda_n} \uparrow \frac{1}{\|F\|}$,  we have $\varphi_*\left(\frac{(h^*)^0}{\lambda_n w}\right)w \uparrow \varphi_*\left(\frac{(h^*)^0}{\|F\| w}\right)w$ a.e. So by the Monotone Convergence Theorem, 
	\[
	\lim_{n \rightarrow \infty}P_{\varphi_*,w}\left(\frac{h}{\lambda_n}\right)  = \lim_{n\rightarrow \infty}\int_I\varphi_*\left(\frac{(h^*)^0}{\lambda_n w}\right)w = \int_I \varphi_*\left(\frac{(h^*)^0}{\|F\| w}\right)w = P_{\varphi_*, w}\left(\frac{h}{\|F\|}\right),
	\]
	and this proves inequality (\ref{welld}).
	
	Since $\rho_{\varphi,w}(kf) < \infty$, $S(kf) \leq \|S\|$ by Theorem \ref{theta}. Now, from (\ref{welld}) we obtain 
	\begin{eqnarray*}
		\frac{1}{k}\left(\rho_{\varphi,w}(kf) + P_{\varphi_*,w}\left(\frac{h}{\|F\|}\right) + \frac{S(kf)}{\|F\|}\right) &\leq& \frac{1}{k}\left(\rho_{\varphi,w}(kf) + P_{\varphi_*,w}\left(\frac{h}{\|F\|}\right) + \frac{\|S\|}{\|F\|}\right)\\
		&\leq& \frac{1}{k}(\rho_{\varphi,w}(kf) + 1).
	\end{eqnarray*}
	Notice that $1= \|f\|_{\varphi,w}^0  =  \frac{1}{k}(\rho_{\varphi,w}(kf) + 1)$ from Theorem \ref{WC}.(ii). This consequently shows that 
	
	\begin{eqnarray*}
		1 = \frac{H(f)}{\|F\|} + \frac{S(f)}{\|F\|} = \frac{\int_I h^*f^*}{\|F\|} + \frac{S(f)}{\|F\|} &=& \frac{\int_I (h^*)^0f^*}{\|F\|} + \frac{S(f)}{\|F\|}\\
		&=& \frac{1}{k}\left(\rho_{\varphi,w}(kf) + P_{\varphi_*,w}\left(\frac{h}{\|F\|}\right) + \frac{S(kf)}{\|F\|}\right)\\
		&=& \frac{1}{k}\left(\rho_{\varphi,w}(kf) + P_{\varphi_*,w}\left(\frac{h}{\|F\|}\right) + \frac{\|S\|}{\|F\|}\right)\\
		&=& \frac{1}{k}(\rho_{\varphi,w}(kf) + 1).
	\end{eqnarray*}
	The sixth and seventh expressions above give us condition (i). Condition (ii) comes from the fifth and sixth expressions. From the second, third, and fourth expressions gives us the first part of condition (iii), and the fourth and fifth expressions show the second part of condition (iii). The proof is finished.
	
	To show the converse, let $f \in \Lambda_{\varphi,w}^0$ such that $\|f\|_{\varphi,w}^0 = 1$ and $k \in K(f)$. Suppose that the conditions (i), (ii) and (iii) are satisfied. Then $1 = \|f\|^0 = \frac{1}{k}(1 + \rho_{\varphi,w}(kf))$ by Theorem \ref{WC}. Let $F = H + S$ be a bounded linear functional on $\Lambda_{\varphi,w}^0$ where $H(f) = \int_I fh$ for some $h \in \mathcal{M}_{\varphi_*,w}$ and $S(f) = 0$ for every $f \in (\Lambda_{\varphi,w}^0)_a$. Hence 
	\small
	\[
	1 = \frac{1}{k}(1 + \rho_{\varphi,w}(kf)) \overset{\rm{(i)}}{=} \frac{1}{k}\left(\rho_{\varphi,w}(kf) + P_{\varphi_*,w}\left(\frac{h}{\|F\|}\right) + \frac{\|S\|}{\|F\|}\right) \overset{\rm{(ii), (iii)}}{=} \frac{1}{k}\left(\int_I \frac{khf}{\|F\|} + \frac{S(kf)}{\|F\|}\right).
	\]
	\normalsize
	So $1 = \frac{1}{k}\left(\int_I \frac{khf}{\|F\|} + \frac{S(kf)}{\|F\|}\right) = \frac{H(f)}{\|F\|}+ \frac{S(f)}{\|F\|} = \frac{F(f)}{\|F\|}$, and the proof is finished.
\end{proof}

By the similar argument, we have the sequence analogue of Theorem \ref{th:suppfunc}.

\begin{Theorem}\label{th:suppseq}
	Let $\varphi$ be an $N$-function and let $F = H + S$ be a bounded linear functional on $\lambda_{\varphi,w}^0$, where $H(x) = \sum_{i = 1}^{\infty} x(i)h(i)$ for some $h \in \mathfrak{m}_{\varphi_*,w}$ and $S(x) = 0$ at $x \in (\lambda_{\varphi,w}^0)_a$. Then $F$ is norm-attaining if and only if there exists $x \in \lambda_{\varphi,w}^0$ with $\|x\|_{\varphi,w}^0 = 1$ such that for some $k \in K(x)$, the following conditions are satisfied 
	\begin{enumerate}[\rm(i)]
		\item $p_{\varphi_*,w}(\frac{h}{\|F\|}) + \frac{\|S\|}{\|F\|} = 1$,
		\item $\|S\| = S(kx)$,
		\item $\sum_{i=1}^{\infty} \frac{kh(i)x(i)}{\|F\|}= \alpha_{\varphi,w}(kx) + p_{\varphi_*,w}(\frac{h}{\|F\|}).$
	\end{enumerate} 
\end{Theorem}

\begin{Proposition}\label{existext}
	Let $\varphi$ be an $N$-function. If $H$ is a bounded linear functional on $(\Lambda_{\varphi,w}^0)_a$, then it has a norm-preserving extension to the whole space $\Lambda_{\varphi,w}^0$, which is also regular. The analogous statement holds for $\lambda_{\varphi,w}^0$.
\end{Proposition}

\begin{proof}
	Since the sequence case can be proven by a similar argument, we only prove for the function case. Let $H$ be a bounded linear functional on $(\Lambda_{\varphi,w}^0)_a$. Then, there exists $h \in \mathcal{M}_{\varphi_*,w}$ such that $H(f) = \int hf$ for $f \in (\Lambda_{\varphi,w}^0)_a$. Without loss of generality, assume $h \geq 0$. Denote by $\tilde{H}$ an extension of $H$ to $\Lambda_{\varphi,w}^0$. Then letting  $\tilde{H}(f) = \int_I hf$ for $f \in \Lambda_{\varphi,w}^0$, we have $|\tilde{H}(f)| = |\int_I hf| \leq \|h\|_{\mathcal{M}_{\varphi,w}}\|f\|^0$ by the H\"older's inequality, so $\tilde{H}$ is  well-defined and bounded on $\Lambda^0_{\varphi,w}$. 
	
	Now, for every $\epsilon > 0$, we can choose $f \in \Lambda_{\varphi,w}^0$ with $\|f\|_{\varphi,w}^0 \le 1$ such that $\|\tilde{H}\| - \frac{\epsilon}{2} < \int_I |hf|$. Define $f_n = |f| \chi_{\{\frac{1}{n} \leq |f| \leq n\}}$, $n\in\mathbb{N}.$ Since $(\Lambda_{\varphi,w}^0)_a = (\Lambda_{\varphi,w}^0)_b$, we see that $f_n \in (\Lambda_{\varphi,w}^0)_a$. Notice also that $f_n \uparrow |f|$ a.e., and so $\lim_{n \rightarrow \infty} \int_I h f_n = \int_I h |f|$ by the Monotone Convergence Theorem. Hence, for all $\epsilon >0$, there exists $N$ such that for every $n \geq N$, $\int_I |hf| < \int_I |hf_n| + \frac{\epsilon}{2}$, and this shows that $\|\tilde{H}\| < \int_I hf_n +\epsilon$. Since $f_n \in (\Lambda_{\varphi,w}^0)_a$ has norm less than or equal to $1$, we have $\|\tilde{H}\| \leq \|H\|$. Therefore, $\|H\| = \|\tilde{H}\|$.  
\end{proof}

\begin{Theorem}\label{ext}
Let $\varphi$ be an $N$-function. If $H$ is a bounded linear functional on $(\Lambda_{\varphi, w}^0)_a$ which attains its norm on on the unit ball of $(\Lambda_{\varphi, w}^0)_a$, then $H$ has a unique norm-preserving extension to $\Lambda_{\varphi,w}^0$, which is also regular. The analogous statement holds for $\lambda_{\varphi,w}^0$.
\end{Theorem}

\begin{proof}
	
	The proof will be given only for function space. The existence of a norm-preserving extension of a bounded linear functional $H$ on $(\Lambda_{\varphi,w}^0)_a$ to the whole space $\Lambda_{\varphi,w}^0$, denoted by $\tilde{H}$, has been shown in Proposition \ref{existext}.
	
	First we show that this extension is unique among regular functionals. Indeed, suppose that we have another norm-preserving extension of $H$, say $\tilde{G}$. Then for $f \in (\Lambda_{\varphi,w}^0)_a$, we have $0 = H(f) - H(f) = (\tilde{H}-\tilde{G})(f)$. Since $\tilde{H}$ and $\tilde{G}$ are regular functionals, the only possibility is when $\tilde{H}=\tilde{G}$. 
	
	Now we show that none of the functionals  $H+ S$, where $H$ is a regular part and a singular part $S\neq 0$, is a norm-preserving extension. Without loss of generality, assume that $\|H\| = \|h\|_{\mathcal{M}_{\varphi_*,w}} = 1$ for some $h \in \mathcal{M}_{\varphi_*,w}$. Since $H$ attains its norm  there exists $f \in (\Lambda_{\varphi,w}^0)_a$ with $\|f\|_{\varphi,w} = 1$ such that $H(f) = \int_I hf = \|h\|_{\mathcal{M}_{\varphi_*,w}}$. In view of Theorem \ref{th:suppfunc}, we have $P_{\varphi_*,w}(h) = 1$. Now, define the function $g(\lambda) = P_{\varphi_*,w}(\frac{h}{\lambda}) + \frac{1}{\lambda} \|S\|$, $\lambda > 0$. Note that the function $g(\lambda)$ is decreasing and continuous on the interval $(1, \infty)$ and right-continuous at $\lambda = 1$. From the fact that $g(1) = P_{\varphi_*,w}(h) + \|S\| = 1 + \|S\| > 1$, there exists $\lambda_0 > 1$ such that $P_{\varphi_*,w}(\frac{h}{\lambda_0}) + \frac{1}{\lambda_0}\|S\| > 1$. But then, this implies that $\|H+S\| \geq \lambda_0 >1 = \|H\|$ by Theorem \ref{Orlicz}. Thus, if $S \neq 0$, $H+S$ is not norm-preserving, so $\tilde{H}$ is the only norm-preserving extension of $H$ to $\Lambda_{\varphi,w}^0$.     
\end{proof}

\begin{Theorem}\label{everyext}
	Let $\varphi$ be an $N$-function. The following statements are equivalent:
	\begin{enumerate}[\rm(i)]
		\item Every bounded linear functional $H$ on $(\Lambda_{\varphi,w}^0)_a$ has a unique norm-preserving extension to $\Lambda_{\varphi,w}^0$;
		\item $\varphi$ or $\varphi_*$ satisfies the appropriate $\Delta_2$-condition.
	\end{enumerate}
	The analogous statement holds for $\lambda_{\varphi,w}^0$.
\end{Theorem}

\begin{proof}
	Since the proof for the sequence case is almost identical, we only provide the proof for the function space. 
	
	(i)$\implies$(ii) Assume that $\varphi$ and $\varphi_*$ do not satisfy the appropriate $\Delta_2$-condition. By the first assumption, there is a bounded linear functional $F = H+S$ on $\Lambda_{\varphi,w}^0$, where $S \neq 0$,  and $H$ is a regular functional. Now, from the fact that $\varphi_*$ does not satisfy $\Delta_2$-condition, we can choose $H$ generated by $h \in \mathcal{M}_{\varphi_*,w}$ such that $\|H\| = \|h\|_{\mathcal{M}_{\varphi_*,w}} =1$ and $P_{\varphi_*,w}(h) < 1$ by Lemma \ref{cor:notdelta2}. Choose $S \neq 0$ such that $\|S\| = 1 - P_{\varphi_*,w}(h)$. Then we obtain $P_{\varphi_*,w}(h) + \|S\| =1$. Define the function $f(\lambda) = P_{\varphi_*,w}(\frac{h}{\lambda}) + \frac{1}{\lambda}\|S\|$, $\lambda> 0$. The function $f(\lambda)$ is strictly decreasing, continuous on $(1, \infty)$, right-continuous at $\lambda = 1$, and $f(1) = 1$. In view of Theorem \ref{Orlicz}, observe that $1 \geq f(\|H+S\|) \geq f(1) = f(\|H\|)= 1$. Hence we have $\|H+S\| = \|H\|$. Let $\tilde{H}$ be an extension of a bounded linear functional $H$ on $(\Lambda_{\varphi,w}^0)_a$ to the whole space, as given in Proposition \ref{existext}. Then $\|H + S\| = \|H\| = \|\tilde{H}\|$. However, $H + S \neq \tilde{H}$ because $S \neq 0$, so we have two distinct norm-preserving extensions of $H$ in this case. 
	
	(ii)$\implies$(i) If $\varphi$ satisfies the appropriate $\Delta_2$-condition, we have $(\Lambda_{\varphi,w}^0)_a = \Lambda_{\varphi,w}^0$, so there is nothing to prove \cite{K}. Hence we only consider when $\varphi$ does not satisfy appropriate $\Delta_2$-condition but $\varphi_*$ does. Assume that $H$ is a bounded linear functional on $(\Lambda_{\varphi,w}^0)_a$ such that $\|H\| =\|h\|_{\mathcal{M}_{\varphi_*,w}} = 1$. Let $\tilde{H}$ be an extension of $H$ to $\Lambda_{\varphi,w}^0$ as given in Proposition \ref{existext}. From such an extension, we also have $\|\tilde{H}\| = \|h\|_{\mathcal{M}_{\varphi_*,w}} =1$. Now, we show that $P_{\varphi_*,w}(h) = 1$. Assume for the contrary that $P_{\varphi_*,w}(h) < 1$. Define a function $g(\lambda) = P_{\varphi_*,w}(\frac{h}{\lambda})$. In view of Theorem \ref{Sep} and from the fact that  $\varphi_*$ satisfies the appropriate $\Delta_2$-condition, the function $g$ is continuous on $(0, \infty)$. Note that $g(\lambda)$ is a strictly decreasing function. Then there exists $\lambda_0 < 1$ such that $P_{\varphi_*,w}(\frac{h}{\lambda_0}) =1$, which is a contradiction to the fact that $\|h\|_{\mathcal{M}_{\varphi,w}} = 1$. 
	
	If $S \neq 0$, we have $1< P_{\varphi_*,w}(h) + \|S\| = 1 + \|S\|$. The function $f(\lambda) = P_{\varphi_*,w}\left(\frac{h}{\lambda}\right) + \frac{1}{\lambda}\|S\|$ is continuous on $(0, \infty)$ by Theorem \ref{Sep} and by the fact that $\varphi_*$ satisfies the appropriate $\Delta_2$-condition. Since $f(1) >1$, there exists $\lambda_0 > 1$ such that $P_{\varphi_*,w}\left(\frac{h}{\lambda_0}\right) + \frac{1}{\lambda_0}\|S\| > 1$. But then, this implies $\|H + S\| \geq \lambda_0 > 1 = \|H\|$  by Theorem \ref{Orlicz}. Thus, $\tilde{H}$ is the only norm-preserving extension of $H$ to $\Lambda_{\varphi,w}^0$.
\end{proof}

\end{document}